\documentclass{amsart}
\usepackage{amssymb, amsmath, amsthm, graphics, comment, xspace, enumerate}
\usepackage{hyperref}
\usepackage{color}
\newtheorem{thm}{Theorem}[section]

\newtheorem{prop}[thm]{Proposition}
\theoremstyle{definition}
\newtheorem{defn}[thm]{Definition}
\newtheorem{example}[thm]{Example}
\theoremstyle{remark}
\newtheorem{rem}[thm]{Remark}
\numberwithin{equation}{section}

\begin{document}
\title[Abstract non-scalar Volterra difference equations]{Abstract non-scalar Volterra difference equations}

\author{Marko Kosti\' c}
\address{Faculty of Technical Sciences,
University of Novi Sad,
Trg D. Obradovi\' ca 6, 21125 Novi Sad, Serbia}
\email{marco.s@verat.net}

{\renewcommand{\thefootnote}{} \footnote{
\\ \text{  }  \ \    {\it Key words and phrases.} Abstract non-scalar Volterra difference equations, abstract multi-term fractional difference equations, Weyl fractional derivatives, Poisson transform, almost periodic solutions.
\\  \text{  }
This research is partially supported by grant 451-03-68/2020/14/200156 of Ministry
of Science and Technological Development, Republic of Serbia.}}

\begin{abstract}
In this paper, we
investigate the abstract non-scalar Volterra difference equations. We employ the Poisson like transforms to connect the solutions of the abstract non-scalar Volterra integro-differential equations and the abstract non-scalar Volterra difference equations. We also investigate the existence, uniqueness and almost periodicity of solutions to the abstract multi-term fractional difference equations with Weyl fractional derivatives.
\end{abstract}
\maketitle

\section{Introduction and preliminaries}

The theory of abstract Volterra integro-differential equations is still a very active field of research of many mathematicians. Fractional calculus and fractional differential equations, which have emerged as the most important and prominent areas of applied mathematics in the last three decades or so, can be viewed as a special part of this theory (cf. the monographs \cite{knjiga,knjigaho,FKP,prus} and references cited therein for more details on the subject).

Discrete fractional calculus, discrete fractional equations and discrete Volterra equations  have received much attention recently, as well (cf. the  monographs \cite{abbasi} by S. Abbas et al., \cite{annaby} by M. H. Annaby, Z. S. Mansour, \cite{fere} by R. A. C. Ferreira, \cite{dfc} by
C. Goodrich, A. C. Peterson and the research articles \cite{abadias,abadias1,alvarez,atici1,atici2,atici3,cao,diblik}, \cite{mihal,rich-lizama,hamad,whe,holm,afdi,lizama-nach,lizama-pois} for some recent results obtained in this direction). Discrete fractional calculus is incredibly important in modeling of various real phenomena appearing in the theories of neural networks, complex dynamic systems, frequency response analysis, image processing and interval-valued systems, e.g..

On the other hand,
the abstract higher-order difference equations have been considered in Chapter 13 of the monograph \cite{gil} by M. I. Gil, where it has been assumed that all operator coefficients are bounded linear operators.
In a series of our recent research papers, we have initiated the study of various classes of the abstract difference equations with closed linear operators or multivalued linear operators. For example, in a joint research study \cite{afdi} with  H. C. Koyuncuo\u{g}lu and D. Velinov,
we
have considered the abstract fractional difference
inclusion
\begin{align}\label{obor}
\Delta^{\alpha}_{W}u(k)\in {\mathcal A}u(k+1)+f(k),\quad k\in {\mathbb Z},
\end{align}
where ${\mathcal A}$ is a multivalued linear operator on a complex Banach space $(X,\|\cdot\|)$, $0<\alpha \leq 1$ and $\Delta^{\alpha}_{W}u(k)$ denotes the Weyl fractional
difference operator of order $\alpha$; concerning the existence and uniqueness of almost periodic solutions and almost automorphic solutions of the abstract fractional 
difference equation \eqref{obor} in which ${\mathcal A}=A$ is a bounded linear operator, we would like to mention here the important research articles of C. Lizama and his co-authors \cite{abadias,abadias1,alvarez,afdi}. Further on, in our recent research article \cite{avdi}, we have analyzed the existence and uniqueness of asymptotically almost periodic solutions of the following abstract Volterra difference inclusion
\begin{align*}
u(v)\in f(v)+{\mathcal A}\bigl(a \ast_{0} u\bigr)(v),\quad v\in {\mathbb N}_{0},
\end{align*}
where ${\mathcal A}$ is a multivalued linear operator in $X,$ $(a(v))_{v\in {\mathbb N}_{0}}$ and $(f(v))_{v\in {\mathbb N}_{0}}$ are given sequences in $X.$ In the same paper, we have considered the existence and uniqueness of  almost periodic type solutions and almost automorphic type solutions to
the following abstract Volterra difference inclusion
\begin{align*}
u(k+1)\in {\mathcal A}\sum_{j=-\infty}^{k}a(k-j)u(j+1)+\sum_{j=-\infty}^{k}b(k-j)f(k),\quad k\in {\mathbb Z},
\end{align*}
where ${\mathcal A}$ is a multivalued linear operator in $X,$ $(a(k))_{k\in {\mathbb N}_{0}}$ and $(b(k))_{k\in {\mathbb N}_{0}}$ are given sequences in $X;$
cf. also the research article \cite{keyantuo} by V. Keyantuo et al. for the first steps made in this direction. We have introduced the notion of a discrete $(a,k)$-regularized $C$-resolvent family and provided several results about the generation of discrete $(a,k)$-regularized $C$-resolvent families therein; for more details about the existence and uniqueness of almost periodic type solutions to the abstract Volterra functional-differential-difference equations, we refer the reader to the forthcoming monograph \cite{funkcionalne}. 

Concerning the abstract non-scalar Volterra integral equations, mention should be made of the important monograph \cite{prus} by J. Pr\"uss and the monographs \cite{knjiga,knjigaho,FKP} by M. Kosti\' c. The main aim of this paper is to examine various classes of the abstract non-scalar Volterra difference equations. We use the Poisson like transforms to make a connection between the solutions of the abstract non-scalar Volterra integro-differential equations and the abstract non-scalar Volterra difference equations. 
We consider here the abstract degenerate Volterra difference equation
\begin{equation}\label{prckodem}
Bu(v)=f(v)+\sum_{j=0}^{v}A(v-j)u(j),\quad v\in {\mathbb N}_{0},
\end{equation}
where $f: {\mathbb N}_{0} \rightarrow X$, $B$ is a closed linear operator on $X$ 
and $A: {\mathbb N}_{0}\rightarrow L(Y,X)$, where $Y$ is a Banach space which is continuously embedded in $X,$ as well as the nonautonomous generalization of \eqref{prckodem}. We thoroughly analyze the well-posedness of \eqref{prckodem} in the following particular case:
\begin{align}\label{ces}
A(v)=a_{1}(v)A_{1}+...+a_{n}(v)A_{n},\quad v\in {\mathbb N}_{0},
\end{align}
where $n\in {\mathbb N},$ $A_{1},...,A_{n}$ are closed linear operators on $X,$ and $(a_{1}(v))_{v\in {\mathbb N}_{0}},...,$ $(a_{n}(v))_{v\in {\mathbb N}_{0}}$ are given sequences in $X.$ 
The equation \eqref{prckodem} is a special case of the equation
\begin{equation}\label{prckodem1}
Bu(v)=f(v)+\sum_{i=1}^{n}\sum_{j=0}^{v+v_{i}}A_{i}(v+v_{i}-j)u(j),\quad v\in {\mathbb N}_{0},
\end{equation}
where $v_{1},...,v_{n}\in {\mathbb N}_{0}$ and $A_{1}: {\mathbb N}_{0}\rightarrow L(Y,X),...,\ A_{n}: {\mathbb N}_{0}\rightarrow L(Y,X).$ We will not consider the well-posedness of problem \eqref{prckodem1} in general case and
we will only focus our attention here to the case in which the kernel $(A_{i}(v))_{v\in {\mathbb N}_{0}}$ can be expressed as $A_{i}(\cdot):=a_{i}(\cdot)A_{i},$ $1\leq i\leq n.$ 

In \cite{multi-term},
we have recently analyzed the
abstract multi-term fractional difference equations with Riemann-Liouville and Caputo fractional derivatives. In this paper, we continue the above-mentioned research study by investigating the abstract multi-term fractional difference equations with Weyl fractional derivatives. In contrast to \cite{multi-term}, the domain of solutions of the considered fractional difference equations is the whole integer line ${\mathbb Z};$ moreover, we do not accompany  the considered fractional difference equations with any initial conditions.
Here, we also want to emphasize that the use of sequences $a_{1}(\cdot)=k^{m_{1}-\alpha_{1}}(\cdot),...,$ $a_{n}(\cdot)=k^{m_{n}-\alpha_{n}}(\cdot)$ in \eqref{ces}, where $(k^{\alpha}(v))_{v\in {\mathbb N}_{0}}$ is the C\`esaro sequence, plays a crucial role in our analysis; cf. Section \ref{rema-weyl} for the notion and more details.

The structure of this paper can be briefly described as follows. After explaining the basic notation and terminology used in the paper, we recall the basic facts about the abstract Volterra integral equations of non-scalar type in Subsection \ref{popij}. Discrete $(A,k,B)$-regularized $C$-resolvent families are investigated in Section \ref{seljace}; in Section \ref{debil}, we introduce and systematically analyze various notions of the $(k,C,B,(A_{i})_{1\leq i\leq n})$-solution  operator families connected with the use of kernel $(A(v))_{v\in {\mathbb N}_{0}}$ given by \eqref{ces}. The asymptotically almost periodic type solutions of the abstract multi--term discrete abstract Cauchy problem
\begin{align*}
Bu(v)=f(v)+\sum_{i=1}^{n}A_{i}\bigl(a_{i}\ast_{0} u\bigr)(v+v_{i}),\quad v\in {\mathbb N}_{0},
\end{align*}
where $v_{1},...,v_{n}\in {\mathbb N}_{0},$ are analyzed in Subsection \ref{dacp}.

The abstract multi-term fractional difference equations with Weyl fractional derivatives are investigated in Section \ref{rema-weyl};
here, we also investigate the well-posedness of the following abstract multi-term difference equation
\begin{align*}
Bu(v)=A_{1}\sum_{l=-\infty}^{v+v_{1}}a_{1}(v+v_{1}-l)u(l)+...+A_{n}\sum_{l=-\infty}^{v+v_{n}}a_{n}(v+v_{n}-l)u(l),\quad v\in {\mathbb Z},
\end{align*}
where $B,\ A_{1},...,\ A_{n}$ are closed linear operator on $X.$ In particular, we consider the existence of solutions to the problems
\begin{align*}
\bigl( \Delta^{m} Bu\bigr)(v)&=A_{1}\Bigl(\Delta^{\alpha_{1}}_{W}u\Bigr)(v+v_{1})+...
\\& +A_{n}\Bigl(\Delta^{\alpha_{n}}_{W}u\Bigr)(v+v_{n})+\Delta^{m}(k \circ Cf)(v)+\Delta^{m}g(v),\ v\in {\mathbb Z}
\end{align*}
and
\begin{align*}
B\bigl( \Delta^{m_{n}} h\bigr)(v)&=\sum_{j=1}^{n-1}A_{j}\Bigl(\Delta^{m_{n}-m_{j}}\Delta^{\alpha_{j}}_{W}h\Bigr)(v+v_{j})
\\&+A_{n}\Bigl(\Delta^{\alpha_{n}}_{W}h\Bigr)(v+v_{n})+(k \circ Cf)(v)+g(v),\ v\in {\mathbb Z};
\end{align*}
cf. Theorem \ref{zajeb} for the notation and more details. It is also worthwhile to mention that we connect the solutions of the abstract multi-term fractional differential equation
$$
A_{n}D_{t}^{\alpha_{n}}u(t)+A_{n-1}D_{t}^{\alpha_{n-1}}u(t)+...+A_{1}D_{t}^{\alpha_{1}}u(t)=0,\quad t>0,
$$
where $0\leq \alpha_{1}<\alpha_{2}<...<\alpha_{n}$ and $D_{t}^{\alpha}u(t)$ denotes the Riemann-Liouville fractional derivative of function $u(t)$ of order $\alpha>0,$ with the solutions of the abstract multi-term fractional difference equation
\begin{align*}
A_{n}\Bigl[\Delta^{\alpha_{n}}_{W}u\Bigr](v)&+A_{n-1}\Bigl[\Delta^{\alpha_{n-1}}_{W}u\Bigr]\bigl(v+m_{n}-m_{n-1}\bigr)
\\&+...+A_{1}\Bigl[\Delta^{\alpha_{1}}_{W}u\Bigr]\bigl(v+m_{n}-m_{1}\bigr)=-g(v),\quad v\in {\mathbb Z};
\end{align*}
cf. Example \ref{kucanje} for more details. We also consider certain connections between the abstract multi-term fractional differential equations with Caputo derivatives and the abstract multi-term fractional difference equation of the above form; we particularly analyze the existence of (almost periodic type) solutions to the problem
$$
\Bigl[\Delta^{\alpha}_{W}u\Bigr](v)+Au(v+m)=-g_{\omega}(v),\quad v\in {\mathbb Z},
$$
where $\alpha>0$ and $m=\lceil \alpha \rceil;$ cf. Example \ref{kucanie} for more details.
All considered Volterra difference equations and fractional difference equations can be unsolvable with respect to the highest derivative.
The final section of paper is reserved for the conclusions and final remark about the topics under our consideration.
\vspace{0.1cm}

\noindent {\bf Notation and terminology.} In the sequel, we will always assume that $X$ and $Y$ are two complex Banach spaces such that $Y$ is continuously embedded in $X$ and $C\in L(X).$ If $C$ is injective, then $[R(C)]$ denotes the Banach space $R(C)$ equipped with the norm $\|x\|_{R(C)}:=\|C^{-1}x\|_X$, $x\in R(C)$;
the norm in $X$, resp. $Y$, is denoted by $\|\cdot\|_X$, resp. $\|\cdot\|_Y.$ By the norm continuity we mean the continuity in $L(X)$ and, in many places, we do not distinguish $S(\cdot)$ $(U(\cdot))$ and its restriction to $Y$; $[D(B)] $ denotes the Banach space $D(B)$ equipped with the graph norm. For the notion and more details about the multivalued linear operators, we refer the rader to \cite{FKP}. \index{norm continuity}

The Gamma function is denoted by
$\Gamma(\cdot)$ and the principal branch is always used to take
the powers. Define $g_{\zeta}(t):=t^{\zeta-1}/\Gamma(\zeta)$ and
$0^{\zeta}:=0$ ($\zeta>0,$ $t>0$); $\lceil s\rceil :=\inf\{k\in {\mathbb Z} : k\geq s\}$ ($s\in {\mathbb R}$). Set ${\mathbb N}_{k}:=\{1,....,k\}$ and
${\mathbb N}_{k}^{0}:=\{0,1,....,k\}$ for $k\in {\mathbb N}.$
If the sequences $(a_{k})_{k\in {\mathbb N}_{0}}$ and $(b_{k})_{k\in {\mathbb N}_{0}}$ are given, then we define $(a\ast_{0}b)(\cdot)$ by
$
(a\ast_{0}b)(k):=\sum_{j=0}^{k}a_{k-j}b_{j},\ k\in {\mathbb N}_{0}.
$
Let us recall that the convolution product $\ast_{0}$ is commutative and associative. Further on, if $\alpha >0,$ then the Ces\`aro sequence $
(k^{\alpha}(v))_{v\in {\mathbb N}_{0}}$ is defined by
$
k^{\alpha}(v):=\Gamma(v+\alpha)/(\Gamma(\alpha)v!).
$
It is well-known that, for every $\alpha>0$ and $\beta>0$, we have $k^{\alpha}\ast_{0} k^{\beta}\equiv k^{\alpha+\beta}$ and $| k^{\alpha}(v)-g_{\alpha}(v)|=O(g_{\alpha}(v)|1/v|)$, $v\in {\mathbb N};$ in particular, $ k^{\alpha}(v)\sim g_{\alpha}(v),$ $v\rightarrow +\infty.$ 
Define $k^{0}(0):=1$ and $k^{0}(v):=0,$ $v\in {\mathbb N};$ then  $k^{\alpha}\ast_{0} k^{\beta}\equiv k^{\alpha+\beta}$ for all $\alpha,\ \beta\geq 0.$ 

If the sequences $(a_{k})_{k\in {\mathbb N}_{0}}$ and $(b_{k})_{k\in {\mathbb Z}}$ are given, then we define the Weyl convolution product $(a\circ b)(\cdot)$ by
$$
(a\circ b)(v):=\sum_{l=-\infty}^{v}a(v-l)b(l),\quad v\in {\mathbb Z}.
$$
Due to \cite[Theorem 3.12(ii)-(iii)]{keyantuo}, some logical assumptions ensure that (cf. also the proof of Theorem \ref{zajeb} below):
\begin{align*}
\bigl( f\ast_{0}g\bigr) \circ h =g\circ (f \circ h)=f \circ (g \circ h).
\end{align*}

Suppose that $\alpha>0$, $m=\lceil \alpha \rceil$ and $I=(0, \infty).$ 
The Riemann-Liouville fractional integral of order $\alpha>0$ is defined for any function $f\in L^{1}(I : X),$ by
$
J_t^\alpha f(t):= \bigl(g_\alpha*f\bigr)(t),$ $ t>0.
$
The Riemann-Liouville fractional derivative of order $\alpha>0$ is
defined for any function $f\in L^{1}(I : X)$ with $g_{m
-\alpha}\ast f\in W^{m ,1}(I: X),$ by
\begin{eqnarray*}
D_t^\alpha f(t):=\frac{d^m}{dt^m} \bigl(g_{m-\alpha}*f\bigr)(t)
= D_t^m J_t^{m-\alpha}f(t),\quad t>0,
\end{eqnarray*}
while the Caputo fractional derivative\index{fractional derivatives!Caputo} $\mathbf D_t^{\alpha}u(t)$ is defined for any function $u\in C^{m-1}([0,\infty):X)$ for which $g_{m-\alpha}*(u-\sum_{k=0}^{m-1}u_kg_{k+1})\in C^m([0,\infty):X)$, by
\[
\mathbf D_t^{\alpha}u(t):=\frac{d^m}{dt^m}\bigg[g_{m-\alpha}*\bigg(u-\sum_{k=0}^{m-1}u_kg_{k+1}\bigg)\bigg].
\]

\subsection{Abstract non-scalar Volterra integral equations}\label{popij}

Let $A(t)$ be a locally integrable function from $[0,\tau)$ into $L(Y,X)$. We need to recall the notion introduced in \cite[Definition 2.9.1-Definition 2.9.3]{FKP}:

\begin{defn}\label{def2.1}
Let $k\in C([0,\tau))$ and $k\neq0$, let $\tau\in(0,\infty]$, $f\in C([0,\tau):X)$, and let $A\in L_{loc}^1([0,\tau):L(Y,X))$.
Of concern is the following degenerate Volterra equation:
\begin{equation}\label{1}
Bu(t)=f(t)+\int^t_0A(t-s)u(s)\, ds,\quad t\in[0,\tau).
\end{equation}\index{abstract degenerate non-scalar Volterra equations}
Then a function $u\in C([0,\tau):[D(B)])$ is said to be a strong solution of \eqref{1} if and only if $u\in L^{\infty}_{loc}([0,\tau):Y)$ and \eqref{1} holds on $[0,\tau)$.
\end{defn}

In the following definition, we introduce the main solution concepts for dealing with \eqref{1}:

\begin{defn}\label{def2.2}
Let $\tau\in(0,\infty]$, $k\in C([0,\tau))$, $k\neq0$ and $A\in L_{loc}^1([0,\tau):L(Y,X))$.
A family $(S(t))_{t\in[0,\tau)}$ in $L(X,[D(B)])$ is said to be an $(A,k,B)$-regularized $C$-pseudoresolvent family if and only if the following holds:\index{$(A,k,B)$-regularized $C$-(pseudo)resolvent family}
\begin{itemize}
\item[(S1)$_{c}$] The mappings $t\mapsto S(t)x$, $t\in[0,\tau)$ and $t\mapsto BS(t)x$, $t\in[0,\tau)$ are continuous in $X$ for every $x\in X$, $BS(0)=k(0)C$ and $S(t)C=CS(t)$, $t\in[0,\tau)$.
\item[(S2)] Set $U(t)x:=\int^t_0S(s)x\,ds$, $x\in X$, $t\in[0,\tau)$.
    Then (S2) means $U(t)Y\subseteq Y$, $U(t)_{\mid Y}\in L(Y)$, $t\in[0,\tau)$ and $(U(t)_{\mid Y})_{t\in[0,\tau)}$ is locally Lipschitz continuous in $L(Y)$.\index{function!Lipschitz continuous}
\item[(S3)] The resolvent equations
\begin{equation}\label{3}
BS(t)y=k(t)Cy+\int^t_0A(t-s)dU(s)y,\quad t\in[0,\tau),\ y\in Y,
\end{equation}
\begin{equation}\label{4}
BS(t)y=k(t)Cy+\int^t_0S(t-s)A(s)y\,ds,\quad t\in[0,\tau),\ y\in Y,
\end{equation}
hold; \eqref{3}, resp. \eqref{4}, is called the first resolvent equation, resp. the second resolvent equation.
\end{itemize}
An $(A,k,B)$-regularized $C$-pseudoresolvent family $(S(t))_{t\in[0,\tau)}$ is said to be an $(A,k,B)$-regularized $C$-resolvent family if additionally:\index{$(A,k,B)$-regularized $C$-resolvent family}
\begin{itemize}
\item[(S4)] For every $y\in Y$, $S(\cdot)y\in L^{\infty}_{loc}([0,\tau):Y)$.
\end{itemize}
An operator family $(S(t))_{t\in[0,\tau)}$ in $L(X,[D(B)])$ is called a weak $(A,k,B)$-reg\-ularized $C$-pseudoresolvent family if and only if (S1)$_{c}$ and\index{$(A,k,B)$-regularized $C$-(pseudo)resolvent family!weak} \eqref{4} hold.
\end{defn}

The condition (S3) can be rewritten in the following equivalent form:
\begin{align*}
\begin{split}
\text{(S3)'}\qquad\qquad &BU(t)y=\Theta(t)Cy+\int^t_0A(t-s)U(s)y\,ds,\quad t\in[0,\tau),\ y\in Y,\\
&BU(t)y=\Theta(t)Cy+\int^t_0U(t-s)A(s)y\,ds,\quad t\in[0,\tau),\ y\in Y.
\end{split}
\end{align*}

We also need the following notion:

\begin{defn}\label{uniq-prim}
Let $\tau\in(0,\infty]$, $k\in C([0,\tau))$, $k\neq0$ and $A\in L_{loc}^1([0,\tau):L(Y,X))$.
A strongly continuous operator family $(V(t))_{t\in[0,\tau)}\subseteq L(X)$ is said to be an $(A,k,B)$-regularized $C$-uniqueness family if and only if\index{$(A,k,B)$-regularized $C$-uniqueness family}
\[
V(t)By=k(t)Cy+\int^t_0V(t-s)A(s)y\,ds,\quad t\in[0,\tau),\ y\in Y\cap D(B).
\]
\end{defn}

\section{Discrete $(A,k,B)$-regularized $C$-resolvent families}\label{seljace}

Suppose that $B(k)$ is a closed linear operator acting in $X$ ($k\in {\mathbb N}_{0}$),
and $A: {\mathbb N}_{0}\rightarrow L(Y,X)$. We would like to introduce the following notion:

\begin{defn}\label{def2.1}
\begin{itemize}
\item[(i)] Let $f: {\mathbb N}_{0} \rightarrow X$.
Then it is said that a sequence $(u_{k})_{k\in {\mathbb N}_{0}}$ is a solution of the abstract degenerate Volterra difference equation
\begin{equation}\label{123123}
B(v)u(v)=f(v)+\sum_{j=0}^{v}A(v-j)u(j),\quad v\in {\mathbb N}_{0}
\end{equation}\index{abstract degenerate non-scalar Volterra difference equations}
if and only if $u(j)\in D(B(j)) \cap Y$ for all $j,\ v\in {\mathbb N}_{0}$,  and \eqref{123123} holds.
\item[(ii)] Let $k: {\mathbb N}_{0} \rightarrow {\mathbb C}$ and $k\neq 0.$ Then a family $(S(v))_{v\in {\mathbb N}_{0}}$ in $L(X)$ is said to be a discrete $(A,k,B)$-regularized $C$-resolvent family if and only if $S(v)C=CS(v)$, $v\in {\mathbb N}_{0}$ and the following holds:\index{discrete $(A,k,B)$-regularized $C$-resolvent family}
\begin{itemize}
\item[(S1)] $S(v)Y\subseteq Y$, $S(v)_{\mid Y}\in L(Y)$, $v\in {\mathbb N}_{0}$,
\begin{equation}\label{3123}
B(v)S(v)y=k(v)Cy+\sum_{j=0}^{v}A(v-j)S(j)y,\quad v\in {\mathbb N}_{0},\ y\in Y,
\end{equation}
and
\begin{equation}\label{4123}
B(v)S(v)y=k(v)Cy+\sum_{j=0}^{v}S(v-j)A(j)y,\quad v\in {\mathbb N}_{0},\ y\in Y;
\end{equation}
\eqref{3123}, resp. \eqref{4123}, is called the first resolvent equation, resp. the second resolvent equation.
\end{itemize}
An operator family $(S(v))_{v\in {\mathbb N}_{0}}$ in $L(X)$ is called a discrete weak $(A,k,B)$-regularized $C$-resolvent family if and only if $S(v)C=CS(v)$, $v\in {\mathbb N}_{0}$, $S(v)Y\subseteq Y$, $S(v)_{\mid Y}\in L(Y)$, $v\in {\mathbb N}_{0}$ and\index{discrete $(A,k,B)$-regularized $C$-(pseudo)resolvent family!weak} \eqref{4123} hold.
\item[(iii)] Let $k: {\mathbb N}_{0} \rightarrow {\mathbb C}$ and $k\neq 0.$ 
Then a strongly continuous operator family $(W(v))_{v\in {\mathbb N}_{0}}\subseteq L(X)$ is said to be a discrete $(A,k,B)$-regularized $C$-uniqueness family if and only if, for every $v\in {\mathbb N}_{0},$ we have \index{discrete $(A,k,B)$-regularized $C$-uniqueness family}
\[
W(v)B(v)y=k(v)Cy+\sum_{j=0}^{v}W(v-j)A(j)y,\quad y\in Y\cap D(B(v)).
\]
\end{itemize}
\end{defn}

The notion introduced in Definition \ref{def2.1}(ii)-(iii) has not been considered elsewhere even in the case that $B={\rm I}.$
Further on,
it is clear that the assumptions $S(v)Y\subseteq Y$, $S(v)_{\mid Y}\in L(Y)$, $v\in {\mathbb N}_{0}$ and $A(j)S(v)y=S(v)A(j)y$ for all $j,\ v\in {\mathbb N}_{0}$ and $y\in Y$ ensure that \eqref{4123} implies \eqref{3123}. We feel it is our duty to say that the equation \eqref{3123} is only an unsatisfactory attempt to provide a discrete analogue of the first resolvent equation \eqref{3} as well as that the last mentioned equation can be also considered in the nonautonomous setting (wth certain difficulties).

In the following proposition, we analyze the uniqueness of solutions to \eqref{123123}:

\begin{prop}\label{profice}
Suppose that a sequence $(u_{k})_{k\in {\mathbb N}_{0}}$ is a solution of \eqref{123123} and the operator $B(v)-A(0)$ is injective for all $v\in {\mathbb N}_{0}.$ Then $u(\cdot)$ is uniquely determined.
\end{prop}

\begin{proof}
Suppose that $f\equiv 0$ in \eqref{123123}. Then it is clear that \eqref{123123} with $v=0$ gives $B(0)u(0)=A(0)u(0),$ so that $u(0)$ is uniquely determined. If $v=1,$ then we get
$B(1)u(1)=A(1)u(0)+A(0)u(1),$ so that $u(1)$ is uniquely determined, as well.
Proceeding in this way and using the injectiveness of the operator $B(v)-A(0)$ for all $v\in {\mathbb N}_{0},$ we simply obtain the required statement. 
\end{proof}

Further on, if $B(v)\equiv B,$ $g: {\mathbb N}_{0} \rightarrow {\mathbb C}$ and $g\neq 0,$ then we define $S_{g}(v)x:=(g\ast_{0} S)(v)x,$ $v\in {\mathbb N}_{0}$ and $x\in X.$ It can be simply shown that, if $(S(v))_{v\in {\mathbb N}_{0}}\subseteq L(X)$ is a discrete (weak) $(A,k,B)$-regularized $C$-resolvent family [$(W(v))_{v\in {\mathbb N}_{0}}\subseteq L(X)$ is a discrete $(A,k,B)$-regularized $C$-uniqueness family], then $(S_{g}(v))_{v\in {\mathbb N}_{0}}\subseteq L(X)$ [$(W_{g}(v))_{v\in {\mathbb N}_{0}}\subseteq L(X)$] is a discrete (weak) $(A,k\ast_{0}g,B)$-regularized $C$-resolvent family [discrete $(A,k\ast_{0}g,B)$-regularized $C$-uniqueness family]. 

Concerning the notion of a discrete $(A,k,B)$-regularized $C$-uniqueness family, we would like to state the following result (cf. also \cite[Proposition 2.9.5(i)]{FKP}; the interested reader may try to transfer the statement of \cite[Proposition 2.9.4(i)]{FKP} to the discrete solution operator families, as well):

\begin{prop}\label{ref}
Suppose that $k: {\mathbb N}_{0} \rightarrow {\mathbb C}$, $k\neq 0$, a sequence $(u_{k})_{k\in {\mathbb N}_{0}}$ is a solution of the abstract degenerate Volterra difference equation \eqref{123123}
and a strongly continuous operator family $(W(v))_{v\in {\mathbb N}_{0}}\subseteq L(X)$ is a discrete $(A,k,B)$-regularized $C$-uniqueness family.
\begin{itemize}
\item[(i)] Then $W(0)f(0)=k(0)Cu(0).$
\item[(ii)] If $B(v)\equiv B,$ then we have $(kC\ast_{0} u)(v)=(W\ast_{0}f)(v)$ for all $v\in {\mathbb N}_{0}.$ In particular, if $k(0)\neq 0$ and $C$ is injective, then there exists at most one solution of \eqref{123123}.
\end{itemize}
\end{prop}

\begin{proof}
The given assumptions imply
$$
W(0)B(0)u(0)=k(0)Cu(0)+W(0)A(0)u(0)\ \ \mbox{ and }\ \ B(0)f(0)=f(0)+A(0)u(0).
$$
This simply yields $W(0)f(0)=k(0)Cu(0).$ To prove (ii) with $B(v)\equiv B,$ observe first that{\small
$$
W(1)Bu(0)=k(1)Cu(0)+W(1)A(0)u(0)+W(0)A(1)u(0) \mbox{ and }Bu(0)=f(0)+A(0)u(0).
$$}
Hence,
\begin{align}\label{hjk}
W(1)f(0)=k(1)Cu(0)+W(0)A(1)u(0).
\end{align}
Similarly, we have
$$
W(0)Bu(1)=k(0)Cu(1)+W(0)A(0)u(1) \mbox{ and } Bu(1)=f(1)+A(1)u(0)+A(0)u(1).
$$
Hence,
\begin{align}\label{hjk1}
W(0)f(1)=k(0)Cu(1)-W(0)A(1)u(0).
\end{align}
Keeping in mind \eqref{hjk}-\eqref{hjk1}, we get that $(kC\ast_{0} u)(1)=(W\ast_{0}f)(1).$ Considering the terms $W(v)Bu(0),$ $W(v-1)Bu(1),...,$ and $W(0)Bu(v),$ a similar line of reasoning gives that $(kC\ast_{0} u)(v)=(W\ast_{0}f)(v)$ for all $v\in {\mathbb N}.$ It is clear that, if $f\equiv 0$, $k(0)\neq 0$ and $C$ is injective, then $(kC\ast_{0} u)(v)=0$ for all $v\in {\mathbb N}_{0};$ inductively, it readily follows that $u\equiv 0.$
\end{proof}

We continue by stating the following result (cf. also Theorem \ref{zajeb} with $v_{1}=...=v_{n}=0$):

\begin{thm}\label{prusonja}
Suppose that $k: {\mathbb N}_{0} \rightarrow {\mathbb C},$ $k\neq 0$ and $(S(v))_{v\in {\mathbb N}_{0}}\subseteq L(X)$ is an operator family such that $S(v)Y\subseteq Y$, $S(v)_{\mid Y}\in L(Y)$, $v\in {\mathbb N}_{0}$, \eqref{3123} holds (this, in particular, holds if $(S(v))_{v\in {\mathbb N}_{0}}$ is a discrete $(A,k,B)$-regularized $C$-resolvent family), $\sum_{v=0}^{+\infty}\| S(v)_{\mid Y}\|_{L(Y)}<+\infty$ and $\sum_{v=0}^{+\infty}\| [A\ast_{0} S](v)\|_{L(Y,X)}<+\infty.$ Suppose, further, that \emph{(i)} or \emph{(ii)} holds, where:
\begin{itemize}
\item[(i)] $f: {\mathbb Z} \rightarrow Y$ is a bounded sequence, $k\in l^{1}({\mathbb Z} :Y)$  and $\sum_{v=0}^{+\infty}\| A(v)\|_{L(Y,X)}<+\infty ;$
\item[(ii)] $f\in l^{1}({\mathbb Z} :Y),$ $k: {\mathbb Z} \rightarrow Y$ is a bounded sequence and $\sup_{v\geq 0}\| A(v)\|_{L(Y,X)}<+\infty.$
\end{itemize}
Define $u(v):=\sum_{l=-\infty}^{v}S(v-l)f(l),$ $v\in {\mathbb Z}.$ Then $u(\cdot)$ is bounded, $u\in l^{1}({\mathbb Z} :Y)$  provided that 
\emph{(ii)} holds, and 
\begin{align}\label{prcko}
B(v)u(v)=(k\circ Cf)(v)+\sum_{l=-\infty}^{v}A(v-l)u(l),\quad v\in {\mathbb Z}.
\end{align}
\end{thm}

\begin{proof}
It is clear that the sequence $(k \circ Cf)(\cdot) $ is well-defined, $u(\cdot)$ is well-defined and bounded if (i) holds as well as that $u\in l^{1}({\mathbb Z} : X)$ if (b) holds. The functional equality (S1) and the closedness of the linear operator $B(v)$ in combination with the assumptions $\sum_{v=0}^{+\infty}\| S(v)_{\mid Y}\|_{L(Y)}<+\infty$ and $\sum_{v=0}^{+\infty}\| [A\ast_{0} S](v)\|_{L(Y,X)}<+\infty$ simply imply the following:
\begin{align*} 
 & Bu(v)-(k \circ Cf)(v)=\sum_{l=-\infty}^{v}BS(v-l)f(l)-(k \circ Cf)(v)
\\&=-(k \circ Cf)(v)+\sum_{l=-\infty}^{v}\Bigl[k(v-l)Cf(l)+\bigl[A\ast_{0} S\bigr](v-l)f(l)\Bigr]
\\& =\sum_{l=-\infty}^{v}\bigl[A\ast_{0} S\bigr](v-l)f(l)=\sum_{s=0}^{+\infty}\bigl[A\ast_{0} S\bigr](s)f(v-s)
\\& =\sum_{s=0}^{+\infty}A(s)\sum_{r=0}^{+\infty}S(r)f(v-s-r)=\sum_{s=0}^{+\infty}A(s)u(v-s)=\sum_{l=-\infty}^{v}A(v-l)u(l),\ v\in {\mathbb Z},
\end{align*}
where we have also used the Fubini theorem in the last line of computation.
\end{proof}

The main result of this section reads as follows:

\begin{thm}\label{stru}
Suppose that $a\in {\mathbb R},$ $\omega \in {\mathbb R}\setminus \{0\},$ $k: {\mathbb N}_{0} \rightarrow {\mathbb C}$ and $k\neq 0,$ $(S(t))_{t\geq 0}$ in $L(X,[D(B)])$ is a weak $(A,k,B)$-regularized $C$-pseudoresolvent family and $(W(t))_{t\geq 0}\subseteq L(X)$ is an $(A,k,B)$-regularized $C$-uniqueness family. Suppose further that there exist finite real constants $M\geq 1$ and $\epsilon>0$ such that $\| A(t)\|_{L(Y,X)}+\|S(t)\|_{L(X,[D(B)])}+\|W(t)\|_{L(X)}\leq M\exp(t(a-\epsilon)),$ $t\geq 0.$
Define
\begin{align*}
k_{a,w}(v):=\int^{+\infty}_{0}e^{-at}\frac{(\omega t)^{v}}{v!}k(t)\, dt,\quad v\in {\mathbb N}_{0},
\end{align*},
\begin{align*}
A_{a,w}(v)y:=\int^{+\infty}_{0}e^{-at}\frac{(\omega t)^{v}}{v!}A(t)y\, dt,\quad v\in {\mathbb N}_{0},\ y\in Y,
\end{align*}
\begin{align*}
S_{a,w}(v)x:=\int^{+\infty}_{0}e^{-at}\frac{(\omega t)^{v}}{v!}S(t)x\, dt,\quad v\in {\mathbb N}_{0},\ x\in X
\end{align*}
and
\begin{align*}
W_{a,w}(v)x:=\int^{+\infty}_{0}e^{-at}\frac{(\omega t)^{v}}{v!}W(t)x\, dt,\quad v\in {\mathbb N}_{0},\ x\in X.
\end{align*}
Then $(S_{a,w}(v))_{v\in {\mathbb N}_{0}}$ is a discrete weak $(A,k_{a,w},B)$-regularized $C$-resolvent family and $(W_{a,w}(v))_{v\in {\mathbb N}_{0}}$ is a discrete $(A,k_{a,w},B)$-regularized $C$-uniqueness family. Furthermore, if $(S(t))_{t\geq 0}$ is an $(A,k,B)$-regularized $C$-pseudoresolvent family, then for any $v\in {\mathbb N}$ and $y\in Y,$ we have
\begin{align}
BS_{a,w}(v)y=k_{a,\omega}(v)Cy+a\omega^{v}\bigl( A_{a,w} \ast_{0} U\bigr)(v)y-\omega^{v}(-1)^{v}\bigl( A_{a,w} \ast_{0} U\bigr)(v-1)y.
\end{align}
\end{thm}

\begin{proof}
Suppose that $v\in {\mathbb N}_{0}$ and $y\in Y$ are fixed. Applying the second resovent equation, the basic operational properties of the vector-valued Laplace transform and the Leibniz rule, it readily follows that (cf. also the proof of \cite[Theorem 3.4]{lizama-pois}):
\begin{align*}
& BS_{a,w}(v)y=\int^{+\infty}_{0}e^{-at}\frac{(\omega t)^{v}}{v!}S(t)y\, dt=\int^{+\infty}_{0}e^{-at}\frac{(\omega t)^{v}}{v!}BS(t)y\, dt
\\& =\int^{+\infty}_{0}e^{-at}\frac{(\omega t)^{v}}{v!}\Biggl[k(t)Cy+\int^{t}_{0}S(t-s)A(s)y\, ds\Biggr]\, dt
\\& = k_{a,w}(v)Cy+(-1)^{v}\frac{\omega^{v}}{v!}\Biggl[\hat{S}(\lambda) \hat{A}(\lambda)\Biggr]^{(v)}_{\lambda=a}
\\& =k_{a,w}(v)Cy+(-1)^{v}\frac{\omega^{v-j}\omega^{j}}{v!}\sum_{j=0}^{v}\binom{v}{j}\hat{S}^{(v-j)}(a)\hat{A}^{(j)}(a)
\\& =k_{a,w}(v)Cy+\frac{\omega^{v-j}\omega^{j}}{v!}\sum_{j=0}^{v}\binom{v}{j}\int^{+\infty}_{0}e^{-at}t^{v-j}S(t)\Biggl[ \int^{+\infty}_{0}e^{-ar}r^{v-j}A(r)y\, dr\Biggr]\, dt
\\& =k_{a,w}(v)Cy+\bigl( S_{a,w} \ast_{0} A_{a,w}\bigr)(v)y,
\end{align*}
as required. The corresponding statement for $(W_{a,w}(v))_{v\in {\mathbb N}_{0}}$ follows similarly; suppose now that $(S(t))_{t\geq 0}$ is an $(A,k,B)$-regularized $C$-pseudoresolvent family. Fix $v\in {\mathbb N}$ and $y\in Y.$ Applying the partial integration and (S3)', we get:
\begin{align*}
& BS_{a,w}(v)y=\int^{+\infty}_{0}e^{-at}\Biggl[a\frac{(\omega t)^{v}}{v!}-\omega \frac{(\omega t)^{v-1}}{(v-1)!}\Biggr] BU(t)y\, dt
\\& =\int^{+\infty}_{0}e^{-at}\Biggl[a\frac{(\omega t)^{v}}{v!}-\omega \frac{(\omega t)^{v-1}}{(v-1)!}\Biggr] \cdot \Biggl\{ \int^{t}_{0}k(s)\, ds \,  Cy+(A\ast U)(t)y \Biggr\}\, dt
\\& =k_{a,w}(v)Cy+\int^{+\infty}_{0}e^{-at}\Biggl[a\frac{(\omega t)^{v}}{v!}-\omega \frac{(\omega t)^{v-1}}{(v-1)!}\Biggr]  (A\ast U)(t)y \, dt.
\end{align*}
Then we can argue as in the first part of the proof to complete it.
\end{proof}

It is worth mentioning that Theorem \ref{stru} can be applied to the abstract non-scalar Volterra integral equations considered by J. Pr\"uss in \cite[Chapter 7--Chapter 9]{prus} (see, especially, the analysis of viscoelastic Timoshenko beam on p. 240) as well as to the abstract degenerate non-scalar Volterra integral equations considered in \cite[Theorem 2.9.7]{FKP} (see, especially, the applications with the abstract differential operators in $L^{p}$-spaces given on pp. 240--241). In such a way, we can consider the well-posedness for certain classes of the abstract semidiscrete non-scalar Volterra difference equations.

Further on, the integrability of $(A,k,B)$-regularized $C$-resolvent families has been considered by J. Pr\"uss in \cite[Section 10.5, pp. 277--281]{prus}, provided that $k\equiv 1$ and $B=C={\rm I}.$ It is worth noting that the condition $\sum_{v=0}^{+\infty}\| S(v)_{\mid Y}\|_{L(Y)}<+\infty$ in Theorem \ref{prusonja} holds for the resolvents $( S(v)\equiv S_{1,1}(v))_{v\in {\mathbb N}_{0}}$ 
obtained by applying Theorem \ref{stru} to the resolvent operator families used in the analysis of Volterra non-scalar equations of variational type (cf. \cite[Corollary 10.7, Corollary 10.8; pp. 280--281]{prus}). In general case of exponentially bounded $(A,k,B)$-regularized $C$-resolvent families, we must use different values of $a,$ which should be sufficiently large, and $w>0$ in order to enure the validity of condition $\sum_{v=0}^{+\infty}\| S(v)_{\mid Y}\|_{L(Y)}<+\infty$ in Theorem \ref{prusonja}. If this is the case, then we can consider the existence and uniqueness of almost periodic type solutions to \eqref{prcko}; the usual almost periodic solutions can be considered if the condition (i) in Theorem \ref{prusonja} holds, while the (metrically) Weyl almost periodic solutions of \eqref{prcko} can be considered if the condition (ii) in Theorem \ref{prusonja} holds (cf. \cite{funkcionalne} for more details).

We close this section by providing some useful observations:

\begin{rem}\label{zas}
Define also 
\begin{align*}
\Theta_{a,w}(v):=\int^{+\infty}_{0}e^{-at}\frac{(\omega t)^{v}}{v!}\int^{t}_{0}k(s)\, ds\, dt,\quad v\in {\mathbb N}_{0}
\end{align*}
and
\begin{align*}
U_{a,w}(v)y:=\int^{+\infty}_{0}e^{-at}\frac{(\omega t)^{v}}{v!}U(t)y\, dt,\quad v\in {\mathbb N}_{0},\ y\in Y.
\end{align*}
Then we have 
$$
BU_{a,w}(v)y=\Theta_{a,w}(v)Cy+\bigl( U_{a,w} \ast_{0} A_{a,w}\bigr)(v)y,\quad v\in {\mathbb N}_{0},\ y\in Y
$$
and
$$
BU_{a,w}(v)y=\Theta_{a,w}(v)Cy+\bigl( A_{a,w} \ast_{0} U_{a,w}\bigr)(v)y,\quad v\in {\mathbb N}_{0},\ y\in Y,
$$
so that, for every $y\in Y$, the sequence $(u(v)\equiv U_{a,w}(v)y)_{v\in {\mathbb N}_{0}}$ is a strong solution of problem \eqref{123123} with $B(v)\equiv B$ and $f(v)\equiv \Theta_{a,w}(v)Cy.$
\end{rem}

\begin{rem}\label{zasav}
In our recent research article \cite{multi-term}, we have analyzed the following Poisson type transform
\begin{align*}
v\mapsto S_{a,w,j}(v)x:=\int^{+\infty}_{0}e^{-(at)^{j}}\frac{(\omega t)^{v}}{v!}S(t)x\, dt,\quad v\in {\mathbb N}_{0},\ x\in X,
\end{align*}
where $j\in {\mathbb N}$ and $j\geq 2.$ It would be very difficult to state an analogue of Theorem \ref{stru} for this transform because we cannot so simply compute the term
$$
\int^{+\infty}_{0}e^{-(at)^{j}}\frac{(\omega t)^{v}}{v!}\Biggl[ \int^{t}_{0}S(t-s)A(s)y\Biggr]\, dt,\quad v\in {\mathbb N}_{0},\ x\in X
$$
in the newly arisen situation.
\end{rem}

\section{$(k,C,B,(A_{i})_{1\leq i\leq n},(v_{i})_{1\leq i\leq n})$-solution operator families}\label{debil}

In this section, we analyze various classes of $(k,C,B,(A_{i})_{1\leq i\leq n},(v_{i})_{1\leq i\leq n})$-solution operator families connected with the use of kernel $(A_{i}(v)\equiv a_{i}(v)A_{i})_{v\in {\mathbb N}_{0}}$ in \eqref{prckodem1}. If $v_{1},...,v_{n}\in {\mathbb N}_{0},$ then we define $v_{max}:=\max(v_{1},...,v_{n})$ and $M:=\{i\in {\mathbb N}_{n} : v_{i}=v_{max}\}.$
We will always assume henceforth that the following condition holds:
\begin{align}\label{kond}
a_{i} : {\mathbb N}_{0} \rightarrow {\mathbb C} \ \ \mbox{ for all }\ \ i\in {\mathbb N}_{n}\ \ \mbox{ and }\ \ a_{i}(0)\neq 0,\quad i\in M.
\end{align}

We continue by introducing the following notion:

\begin{defn}\label{nsbgd}
Suppose that $B$, $A_{1},...,A_{n}$ are closed linear operators on $X$, $C\in L(X)$, $v_{1},...,v_{n}\in {\mathbb N}_{0},$ ${\mathcal I}\subseteq {\mathbb N}_{n},$ $k : {\mathbb N}_{0} \rightarrow {\mathbb C},$ $k\neq 0$ and 
\eqref{kond} holds. Then we say that the operator family $(S(v))_{v\in {\mathbb N}_{0}}\subseteq L(X)$ is a discrete:
\begin{itemize}
\item[(i)] $(k,C,B,(A_{i})_{1\leq i\leq n},(v_{i})_{1\leq i\leq n})$-existence family if and only if the mapping $x\mapsto A_{i}\bigl( a_{i}\ast_{0} S\bigr)(v+v_{i})x,$ $x\in X$ belongs to $L(X)$ for $v\in {\mathbb N}_{0}$, $1\leq i\leq n$ and
\begin{align}\label{bnm}
BS(v)x=k(v)Cx+\sum_{i=1}^{n}A_{i}\bigl( a_{i}\ast_{0} S\bigr)(v+v_{i})x,\quad v\in {\mathbb N}_{0},\ x\in X.
\end{align}
\item[(ii)] $(k,C,B,(A_{i})_{1\leq i\leq n},(v_{i})_{1\leq i\leq n},{\mathcal I})$-existence family if and only if 
$(S(v))_{v\in {\mathbb N}_{0}}$ is $(k,C,B,(A_{i})_{1\leq i\leq n},(v_{i})_{1\leq i\leq n})$-existence family and $S(v)A_{i}\subseteq A_{i}S(v)$ for all $v\in {\mathbb N}_{0}$ and $i\in {\mathbb N}_{n} \setminus {\mathcal I}.$
\end{itemize}
\end{defn}

Any discrete $(k,C,B,(A_{i})_{1\leq i\leq n},(v_{i})_{1\leq i\leq n})$-existence family $(S(v))_{v\in {\mathbb N}_{0}}$ satisfies $(S(v))_{v\in {\mathbb N}_{0}}\subseteq L(X,[D(B)]) .$ 
In the case that $v_{1}=v_{2}=...=v_{n}=0,$ we omit the term ``$(v_{i})_{1\leq i\leq n}$'' from the notation. If this is not the case, then we cannot expect
the uniqueness of $(k,C,B,(A_{i})_{1\leq i\leq n},(v_{i})_{1\leq i\leq n})$-existence family when all terms of this tuple are given in advance. For example, if $n=1,$ $B=C={\rm I},$ $A_{1}=2{\rm I}$, $a_{1}\equiv k\equiv 1$ and $v_{1}=1,$ then we have $S(0)x+2S(1)x+Cx=0,$ so that neither $S(0)$ nor $S(1)$ can be uniquely determined.
The notion of a discrete $(k,C,B,(A_{i})_{1\leq i\leq n})$-existence family looks similar but it is not exactly the same as the notion of a weak discrete  $(A,k,B)$-regularized $C$-resolvent family with $B(\cdot)\equiv B$, $Y:=X$ and $A(v)y:=\sum_{i=1}^{n}a_{i}(v)A_{i}y$ for all $y\in Y$ and $v\in {\mathbb N}_{0}.$ 

We continue by stating the following result:

\begin{prop}\label{prosto}
Suppose that $B$, $A_{1},...,A_{n}$ are closed linear operators on $X$, $C\in L(X)$, $v_{1},...,v_{n}\in {\mathbb N}_{0},$ $k : {\mathbb N}_{0} \rightarrow {\mathbb C},$ $k\neq 0$,
\eqref{kond} holds and $(S(v))_{v\in {\mathbb N}_{0}}\subseteq L(X)$ is a discrete
$(k,C,B,(A_{i})_{1\leq i\leq n},(v_{i})_{1\leq i\leq n})$-existence family. 
If $x\in X,$ $i\in {\mathbb N}_{n}$ and $v_{i}=0$, then  
$S(v)x\in D(A_{i})$ for all $v\in {\mathbb N}_{0}$ and $i\in {\mathbb N}_{n};$ the same holds for each $i\in {\mathbb N}_{n}$ with $v_{i}>0,$ provided that $S(0)x\in D(A_{i}),...,S(v_{i}-1)x\in D(A_{i}).$
\end{prop}

\begin{proof}
We will prove the statement in case $i\in {\mathbb N}_{n}$ and $v_{i}>0$. Then we have $( a_{i}\ast_{0} S)(v_{i})x\in D(A_{i}).$ Since we have assumed $a_{i}(0)\neq 0$ and  $S(0)x\in D(A_{i}),...,S(v_{i}-1)x\in D(A_{i}),$ this simply implies $S(v_{i})x\in D(A_{i}).$ Since $( a_{i}\ast_{0} S)(1+v_{i})x\in D(A_{i}),$ we similarly obtain $S(v_{i}+1)x\in D(A_{i}).$ Proceeding by induction, we get $S(v)x\in D(A_{i})$ for all $v\in {\mathbb N}_{0}.$ 
\end{proof}

Now we will state and prove the following result (cf. also \cite[Theorem 2.5]{avdi}):

\begin{thm}\label{djeca}
Suppose that $B$, $A_{1},...,A_{n}$ are closed linear operators on $X$, $C\in L(X)$ is injective, $k : {\mathbb N}_{0} \rightarrow {\mathbb C},$ $k(0)\neq 0$
and \eqref{kond} holds.
\begin{itemize}
\item[(i)] Suppose, further, that $(S(v))_{v\in {\mathbb N}_{0}}\subseteq L(X)$ is a discrete
$(k,C,B,(A_{i})_{1\leq i\leq n})$-existence family such that $S(0)Bx=BS(0)x$ and $S(0)A_{i}x=A_{i}S(0)x$ for all $x\in D(B) \cap D(A_{1}) \cap ... \cap D(A_{n}).$
Then $(B-\sum_{i=0}^{n}a_{i}(0)A_{i})^{-1}C\in L(X),$
$S(0)=k(0)(B-\sum_{i=0}^{n}a_{i}(0)A_{i})^{-1}C$,
\begin{align}  
\notag S(v)x&=\Biggl(B-\sum_{i=0}^{n}a_{i}(0)A_{i}\Biggr)^{-1}
\\\label{klara}& \times \Biggl[k(v)Cx+\sum_{i=1}^{n}A_{i}\sum_{j=0}^{v-1}a_{i}(v-j)S(j)x\Biggr],\quad v\in {\mathbb N},\ x\in X,
\end{align}
and $A_{i}S(v)\in L(X)$ for all $ i\in {\mathbb N}_{n}$ and $v\in {\mathbb N}_{0}$.
\item[(ii)]
Suppose that $C\in L(X)$ is injective, 
$(B-\sum_{i=0}^{n}a_{i}(0)A_{i})^{-1}C\in L(X)$ and, for every $l\in {\mathbb N}$ and for every choice of integers $a_{j}\in {\mathbb N}_{n}$ for $1\leq j\leq l,$ we have
\begin{align}\label{dinar}
\Biggl[\prod_{j=1}^{l}
\Biggl(B-\sum_{i=0}^{n}a_{i}(0)A_{i}\Biggr)^{-1}A_{a_{j}}\Biggr] \cdot \Biggl(B-\sum_{i=0}^{n}a_{i}(0)A_{i}\Biggr)^{-1}C\in L(X).
\end{align} 
Define $S(0):=k(0)(B-\sum_{i=0}^{n}a_{i}(0)A_{i})^{-1}C$ and $S(v),$ $v\in {\mathbb N},$ recursively by \eqref{klara}.
Then $(S(v))_{v\in {\mathbb N}_{0}}\subseteq L(X)$ is well-defined,
$A_{i}S(v)\in L(X)$ for all $ i\in {\mathbb N}_{n}$, $v\in {\mathbb N}_{0}$ and $(S(v))_{v\in {\mathbb N}_{0}}$ is a unique discrete $(k,C,B,(A_{i})_{1\leq i\leq n})$-existence family. Furthermore, if ${\mathcal I}\subseteq {\mathbb N}_{n}$ and
\begin{align}\begin{split}\label{univer}
& CB\subseteq BC,\ CA_{i}\subseteq A_{i}C \mbox{ for all }i\in {\mathbb N}_{n}\setminus {\mathcal I},
\\&  \bigl(\forall i\in{\mathbb N}_{n} \setminus {\mathcal I} \bigr)\, \bigl(\forall x\in D(A_{i}) \cap D(B)\bigr)\, Bx\in D(A_{i}),\ A_{i}x\in D(B)\mbox{ and } A_{i}Bx=BA_{i}x, 
\\ &\bigl(\forall i\in{\mathbb N}_{n} \setminus {\mathcal I} \bigr) \bigl(\forall j\in{\mathbb N}_{n} \bigr)\, \bigl(\forall x\in D(A_{i}) \cap D(A_{j})\bigr)
\\& A_{j}x\in D(A_{i}),\ A_{i}x\in D(A_{j})\mbox{ and } A_{i}A_{j}x=A_{j}A_{i}x,
\end{split}
\end{align}
resp.
there exist a closed linear operator $A$ and the complex polynomials $P_{B}(\cdot),$ $ P_{1}(\cdot),...,\ P_{n}(\cdot)$ such that $CA\subseteq AC$ and
$B=P_{B}(A),$ $A_{1}=P_{1}(A),...,\ A_{n}=P_{n}(A),$
then $(S(v))_{v\in {\mathbb N}_{0}}$ is a discrete $(k,C,B,(A_{i})_{1\leq i\leq n},{\mathcal I})$-existence family, resp.  $(S(v))_{v\in {\mathbb N}_{0}}$ is a discrete\\ $(k,C,B,(A_{i})_{1\leq i\leq n},\emptyset)$-existence family.
\item[(iii)] Suppose that $C={\rm I}$, $(B-\sum_{j=0}^{n}a_{j}(0)A_{j})^{-1}\in L(X),$ $\sum_{v=0}^{+\infty}|a_{i}(v)|<+\infty$ for $1\leq i\leq n$, $\sum_{v=0}^{\infty}|k(v)|<+\infty,$ and \emph{(a)} or \emph{(b)} holds, where:
\begin{itemize}
\item[(a)] $A_{i}\in L(X)$ for $1\leq i\leq n$ and 
\begin{align}\label{sda}
1>\sum_{i=1}^{n}\sum_{v=1}^{+\infty}\bigl| a_{i}(v)\bigr| \cdot \Biggl\| \Biggl(B-\sum_{j=0}^{n}a_{j}(0)A_{j}\Biggr)^{-1}A_{i}\Biggr\|. 
\end{align}
\item[(b)] Suppose that $C={\rm I}$, \eqref{univer} holds or there exist a closed linear operator $A$ and the complex polynomials $P_{B}(\cdot),$ $ P_{1}(\cdot),...,\ P_{n}(\cdot)$ such that 
$B=P_{B}(A),$ $A_{1}=P_{1}(A),...,\ A_{n}=P_{n}(A),$ and
\begin{align}\label{sda1}
1>\sum_{i=1}^{n}\sum_{v=1}^{+\infty}\bigl| a_{i}(v)\bigr| \cdot \Biggl\| A_{i}\Biggl(B-\sum_{j=0}^{n}a_{j}(0)A_{j}\Biggr)^{-1}\Biggr\|.
\end{align}
\end{itemize}
Then the requirements in \emph{(ii)} hold and we have
\begin{align}\label{ntc}
\sum_{v=0}^{+\infty}\| S(v)\|<+\infty \ \ \mbox{ and }\ \
\sum_{v=0}^{+\infty}\Bigl\|A_{i}\bigl( a_{i}\ast_{0} S\bigr)(v)\Bigr\|<+\infty \ \ (1\leq i\leq n),
\end{align} 
provided that \emph{(a)} holds, resp.
we have \eqref{ntc} and
\begin{align}\label{potrebno}
\sum_{v=0}^{+\infty}\| A_{i}S(v)\|<+\infty \ \ (1\leq i\leq n),
\end{align}
provided that \emph{(b)} holds.
\end{itemize}
\end{thm}

\begin{proof}
Suppose that the assumptions in (i) hold and $x\in X.$
Then \eqref{bnm} with $v=0$ implies 
$
BS(0)x=k(0)Cx+\sum_{i=1}^{n}a_{i}(0)A_{i} S(0)x
$ so that $(B-\sum_{i=0}^{n}a_{i}(0)A_{i})S(0)x=k(0)Cx$ and $S(0)x\in k(0) (B-\sum_{i=0}^{n}a_{i}(0)A_{i})^{-1}Cx.$ Now we will prove that the multivalued linear operator 
$ (B-\sum_{i=0}^{n}a_{i}(0)A_{i})^{-1}C$ is single-valued so that $(B-\sum_{i=0}^{n}a_{i}(0)A_{i})C\in L(X)$ and $S(0)=k(0)(B-\sum_{i=0}^{n}a_{i}(0)A_{i})^{-1}C.$ In actual fact, let us assume that $y\in (B-\sum_{i=0}^{n}a_{i}(0)A_{i})^{-1}C0$ for some $y\in X.$ This implies $(B-\sum_{i=0}^{n}a_{i}(0)A_{i})y=0$ and $y\in D(B) \cap D(A_{1}) \cap ... \cap D(A_{n})$ since $a_{1}(0)\neq 0,...,a_{n}(0)\neq 0;$ now the prescribed assumption implies that $0=S(0)(B-\sum_{i=0}^{n}a_{i}(0)A_{i})y=(B-\sum_{i=0}^{n}a_{i}(0)A_{i})S(0)y=k(0)Cy.$ Since $k(0)\neq 0$ and $C$ is injective, this simply implies $y=0$, the required conclusions and $A_{i}S(0)\in L(X)$ for all $i\in {\mathbb N}_{n}$ by the closed graph theorem. 

Since $BS(v)\in L(X)$, $A_{i}\bigl( a_{i}\ast_{0} S\bigr)(v)\in L(X)$ for $v\in {\mathbb N}_{0}$, $1\leq i\leq n,$ 
\begin{align}\label{kbc}
BS(v)x=k(v)Cx+\sum_{i=1}^{n}A_{i}\bigl[ a_{i}(0)S(v)x+...+a_{i}(v)S(0)x\Bigr],\quad v\in {\mathbb N}_{0},\ x\in X
\end{align}
and $a_{1}(0)\neq 0,...,a_{n}(0)\neq 0,$
we can inductively prove that $A_{i}S(v)\in L(X)$ for all  $i\in {\mathbb N}_{n}$ and $v\in {\mathbb N}_{0}$. Moreover, \eqref{kbc} easily implies \eqref{klara} after a simple calculation.

In order to prove (ii), let us observe first that
the mappings $S(0)$, $BS(0)$ and $x\mapsto A_{i}\bigl( a_{i}\ast_{0} S\bigr)(0)x,$ $x\in X$ belong to $L(X)$ for $v=0$, $1\leq i\leq n$ as well as that \eqref{bnm} holds with $v=0$ and $A_{j}S(0)\in L(X)$  all $j\in {\mathbb N}_{0},$ since we have assumed that $a_{i}\neq 0$ for all $i\in {\mathbb N}_{n}.$ Now we proceed by induction. Assume that the mappings $S(v'),$ $BS(v')$ and $x\mapsto A_{i}\bigl( a_{i}\ast_{0} S\bigr)(v')x,$ $x\in X$ belong to $L(X)$ for all $v' <v\in {\mathbb N}$, $1\leq i\leq n$ as well as that \eqref{bnm}
holds for all $v'<v$ and $A_{j}S(v')\in L(X)$ for all $j\in {\mathbb N}_{0}.$ The induction hypothesis together with the representation formula \eqref{klara}, the assumption \eqref{dinar} and the closed graph theorem imply that $S(v)\in L(X)$ and $BS(v)\in L(X).$ Since
$$
A_{i}\bigl(a_{1}\ast_{0}S\bigr)(v)x=A_{i}\Biggl[ a_{i}(0)S(v)x+\sum_{j=1}^{v-1}a_{i}(v-j)S(j)x \Biggr],\quad x\in X,\ 1\leq i\leq n,
$$ 
the representation formula \eqref{klara}, the closed graph theorem and the assumption $a_{i}\neq 0$ for all $i\in {\mathbb N}_{0}$ imply that $A_{i}S(v)\in L(X)$ and the mapping $x\mapsto A_{i}\bigl( a_{i}\ast_{0} S\bigr)(v)x,$ $x\in X$ belongs to $L(X)$ for $1\leq i\leq n.$ Keeping in mind \eqref{klara}, the above simply implies \eqref{bnm}, so that the conclusions stated in the first part of theorem hold true. Suppose now that ${\mathcal I}\subseteq {\mathbb N}_{n}$ and the commutation relations in \eqref{univer} hold. We will prove that $(S(v))_{v\in {\mathbb N}_{0}}$ is a discrete $(k,C,B,(A_{i})_{1\leq i\leq n},{\mathcal I})$-existence family, i.e., that for each fixed 
$x\in D(A_{i})$ we have $A_{i}x\in D(S(v))$ and $A_{i}S(v)x=S(v)A_{i}x$ for all $v\in {\mathbb N}_{0}$ and 
$i\in {\mathbb N}_{n} \setminus {\mathcal I}.$ If $v=0,$ then we need to prove that
$$
A_{i}\Biggl(B-\sum_{j=0}^{n}a_{j}(0)A_{j}\Biggr)^{-1}Cx=\Biggl(B-\sum_{j=0}^{n}a_{j}(0)A_{j}\Biggr)^{-1}CA_{i}x,
$$
i.e.,
$$
\Biggl(B-\sum_{i=0}^{n}a_{j}(0)A_{j}\Biggr)A_{i}\Biggl(B-\sum_{j=0}^{n}a_{j}(0)A_{j}\Biggr)^{-1}Cx=CA_{i}x.
$$
 Using the equalities in the second line and the fourth line of \eqref{univer}, the above is equivalent to
$$
A_{i}\Biggl(B-\sum_{i=0}^{n}a_{j}(0)A_{j}\Biggr)\Biggl(B-\sum_{j=0}^{n}a_{j}(0)A_{j}\Biggr)^{-1}Cx=CA_{i}x,
$$
i.e., with $A_{i}Cx=CA_{i}x$, which it has been assumed in the first line of \eqref{univer}. Now we proceed by induction. Suppose that for each $v'<v$, $i\in {\mathbb N}_{n} \setminus {\mathcal I}$ and $x\in D(A_{i})$ we have $A_{i}x\in D(S(v'))$ and $A_{i}S(v')x=S(v')A_{i}x$.
 Let us prove that the last equality holds with $v'=v.$ Using \eqref{klara}, it suffices to show that:
\begin{align*}  
A_{i}&\Biggl(B-\sum_{j=0}^{n}a_{j}(0)A_{j}\Biggr)^{-1}
\Biggl[k(v)Cx+\sum_{j=1}^{n}A_{j}\sum_{l=0}^{v-1}a_{j}(v-l)S(l)x\Biggr]
\\& = \Biggl(B-\sum_{j=0}^{n}a_{j}(0)A_{j}\Biggr)^{-1}
\Biggl[k(v)C+\sum_{j=1}^{n}A_{j}\sum_{l=0}^{v-1}a_{j}(v-l)S(l)\Biggr]A_{i}x.
\end{align*}
The equality of terms containing $k(v)Cx$ is clear, while the equality of the remaining parts follows from the induction hypothesis and the assumption made in the third line and the fourth line od \eqref{univer}, since $\sum_{l=0}^{v-1}a_{j}(v-l)S(l)x\in D(A_{j})$ for all $j\in {\mathbb N}_{n},$ $v\in {\mathbb N}$ and $x\in X.$ It is quite easy to show that, if there exist a closed linear operator $A$ and the complex polynomials $P_{B}(\cdot),$ $ P_{1}(\cdot),...,\ P_{n}(\cdot)$ such that $CA\subseteq AC$ and
$B=P_{B}(A),$ $A_{1}=P_{1}(A),...,\ A_{n}=P_{n}(A),$
then $(S(v))_{v\in {\mathbb N}_{0}}$ is a discrete $(k,C,B,(A_{i})_{1\leq i\leq n},\emptyset)$-existence family.

If $C={\rm I}$ and $(B-\sum_{j=0}^{n}a_{j}(0)A_{j})^{-1}\in L(X),$ then the equation \eqref{dinar} automatically holds by the closed graph theorem.
Suppose now that the assumptions in (iii)(a) hold. 
Then we have
\begin{align*}
\sum_{v=0}^{+\infty}\| S(v)\| &\leq \Biggl\| \Biggl(B-\sum_{i=0}^{n}a_{i}(0)A_{i}\Biggr)^{-1} \Biggr\| \cdot \sum_{v=0}^{+\infty}| k(v)| 
\\& +\sum_{i=1}^{n}\Biggl\| \Biggl(B-\sum_{j=0}^{n}a_{j}(0)A_{j}\Biggr)^{-1}A_{i}\Biggr\|\cdot
\sum_{v=0}^{+\infty}\sum_{j=0}^{v-1}\bigl| a_{i}(v-j)\bigr| \cdot \| S(j)\|  
\\& =\Biggl\| \Biggl(B-\sum_{i=0}^{n}a_{i}(0)A_{i}\Biggr)^{-1} \Biggr\| \cdot \sum_{v=0}^{+\infty}| k(v)| 
\\& +\sum_{i=1}^{n}\sum_{v=1}^{+\infty}\bigl| a_{i}(v)\bigr|\cdot \Biggl\| \Biggl(B-\sum_{j=0}^{n}a_{j}(0)A_{j}\Biggr)^{-1}A_{i}\Biggr\|
 \cdot \sum_{v=0}^{+\infty}\| S(v)\|  .
\end{align*}
Keeping in mind \eqref{sda}, we simply get the first estimate in \eqref{ntc}. The second estimate in \eqref{ntc} holds since $A_{i}\in L(X)$ for $1\leq i\leq n$ and 
$$
\sum_{v=0}^{+\infty}\Bigl\|A_{i}\bigl( a_{i}\ast_{0} S\bigr)(v)\Bigr\| \leq \bigl\|A_{i}\bigr\| \cdot \sum_{v=0}^{+\infty}\bigl| a_{i}(v)\bigr| \cdot \sum_{v=0}^{+\infty}\bigl\| S(v)\bigr\| \ \ (1\leq i\leq n).
$$
Suppose now that the assumptions in (iii)(b) hold. Then we similarly obtain 
\begin{align*}
\sum_{v=0}^{+\infty}\| S(v)\| &\leq \Biggl\| \Biggl(B-\sum_{i=0}^{n}a_{i}(0)A_{i}\Biggr)^{-1} \Biggr\| \cdot \sum_{v=0}^{+\infty}| k(v)| 
\\& +\sum_{i=1}^{n}\Biggl\| A_{i}\Biggl(B-\sum_{j=0}^{n}a_{j}(0)A_{j}\Biggr)^{-1}\Biggr\|\cdot
\sum_{v=0}^{+\infty}\sum_{j=0}^{v-1}\bigl| a_{i}(v-j)\bigr| \cdot \| S(j)\|  
\\& =\Biggl\| \Biggl(B-\sum_{i=0}^{n}a_{i}(0)A_{i}\Biggr)^{-1} \Biggr\| \cdot \sum_{v=0}^{+\infty}| k(v)| 
\\& +\sum_{i=1}^{n}\sum_{v=1}^{+\infty}\bigl| a_{i}(v)\bigr|\cdot \Biggl\| A_{i}\Biggl(B-\sum_{j=0}^{n}a_{j}(0)A_{j}\Biggr)^{-1}\Biggr\|
 \cdot \sum_{v=0}^{+\infty}\| S(v)\|  .
\end{align*}
Now we can use \eqref{sda1} to deduce the first estimate in \eqref{ntc}. Let $1\leq i\leq n;$ then the second estimate in \eqref{ntc} can be proved as follows:
Clearly, we have
$$
\sum_{v=0}^{\infty}\Biggr\|A_{i}\sum_{j=0}^{v}a_{i}(v-j)S(j)\Biggl\| \leq \sum_{v=0}^{+\infty}| a_{i}(v)| \cdot \sum_{v=0}^{+\infty}\| A_{i}S(v)\| 
$$
Therefore, it suffices to show \eqref{potrebno}. But, arguing as above and using our commuting assumptions \eqref{univer}, it follows that
\begin{align*}
\sum_{v=0}^{+\infty}\| A_{i}S(v)\| &\leq \Biggl\| A_{i}\Biggl(B-\sum_{j=0}^{n}a_{j}(0)A_{j}\Biggr)^{-1} \Biggr\| \cdot \sum_{v=0}^{+\infty}| k(v)| 
\\& +\sum_{v=0}^{+\infty}\sum_{j=1}^{n}\Biggl\| A_{i}\Biggl(B-\sum_{j=0}^{n}a_{j}(0)A_{j}\Biggr)^{-1}A_{j}\sum_{l=0}^{v-1}a_{j}(v-l)S(l)\Biggr\|
\\& =\Biggl\| A_{i}\Biggl(B-\sum_{j=0}^{n}a_{j}(0)A_{j}\Biggr)^{-1} \Biggr\| \cdot \sum_{v=0}^{+\infty}| k(v)| 
\\& +\sum_{v=0}^{+\infty}\sum_{j=1}^{n}\Biggl\| A_{j}\Biggl(B-\sum_{j=0}^{n}a_{j}(0)A_{j}\Biggr)^{-1}\sum_{l=0}^{v-1}a_{j}(v-l)A_{i}S(l)\Biggr\|
\\& \leq \Biggl\| A_{i}\Biggl(B-\sum_{j=0}^{n}a_{j}(0)A_{j}\Biggr)^{-1} \Biggr\| \cdot \sum_{v=0}^{+\infty}| k(v)| 
\\& +\sum_{j=1}^{n}\sum_{v=1}^{+\infty}\bigl| a_{j}(v)\bigr|\cdot \Biggl\| A_{j}\Biggl(B-\sum_{j=0}^{n}a_{j}(0)A_{j}\Biggr)^{-1}\Biggr\|\cdot \sum_{v=0}^{+\infty}\| A_{i}S(v)\| .
\end{align*}
Now the required conclusion simply follows from \eqref{sda1}.
\end{proof}

We proceed with some useful observations:

\begin{rem}\label{serija}
\begin{itemize}
\item[(i)] As already clarified,
the equation \eqref{dinar} automatically holds if $C={\rm I};$ moreover, in general  case, we have $C^{l}D_{l}\in L(X),$ where
$D_{l}$ denotes the operator appearing on the left-hand side of \eqref{dinar}.
\item[(ii)] The assumptions made in \eqref{univer} do not imply that the operators $A_{1},...,A_{n}$ are bounded, in general. A simple counterexample can be given with $B=C={\rm I},$ the operators $A_{i}$ being bounded for $i\in {\mathbb N}_{n}\setminus {\mathcal I},$ by assuming also that $R(A_{i}) \subseteq D(A_{j})$ for all $j\in {\mathbb N}_{n}$ and $A_{i}A_{j}x=A_{j}A_{i}x$ for all $j\in {\mathbb N}_{n}$ and $x\in D(A_{j}).$ 
\item[(iii)] It would be very tempting to deduce the estimates \eqref{ntc} and \eqref{potrebno} using the argumentation contained in the proofs of parts (ii)-(iii) of \cite[Theorem 2.5]{avdi}. 
\end{itemize}
\end{rem}

Now we will consider the situation in which $v_{max}>0.$ Set $M:=\{i\in {\mathbb N}_{n} : v_{i}=v_{max}\};$ then we have the following analogue of Theorem \ref{djeca}:  

\begin{thm}\label{djeca-vmax}
\begin{itemize}
\item[(i)] Suppose that $B$, $A_{1},...,A_{n}$ are closed linear operators on $X$, $C\in L(X)$, $k : {\mathbb N}_{0} \rightarrow {\mathbb C},$ $k(0)\neq 0,$
\eqref{kond} holds,
$(S(v))_{v\in {\mathbb N}_{0}}\subseteq L(X)$ is a discrete
$(k,C,B,(A_{i})_{1\leq i\leq n},(v_{i})_{1\leq i\leq n})$-existence family, and for every $i\in M$ and $x\in X,$ we have $S(0)x\in D(A_{i}),...,S(v_{max}-1)x\in D(A_{i})$. Then $S(v)x\in D(A_{i})$ for all $v\in {\mathbb N}_{0},$ $x\in X$, $i\in M,$ 
\begin{align}\label{dokle}
BS(0)x=k(0)Cx+\sum_{i=1}^{n}\bigl(a_{i} \ast_{0} S\bigr)(v_{i})
\end{align}
and $(S(v)$ is a uniquely determined for $v>v_{max}$, provided that the operator $\sum_{i\in M}a_{i}(0)A_{i}$ is injective.

Furthermore, let us assume that \emph{(a)}-\emph{(c)} hold, where:
\begin{itemize}
\item[(a)] $[\sum_{i\in M}a_{i}(0)A_{i}]^{-1}B\in L(X)$ and $[\sum_{i\in M}a_{i}(0)A_{i}]^{-1}C\in L(X);$
\item[(b)] $[\sum_{i\in M}a_{i}(0)A_{i}]^{-1}A_{j}(a_{i}\ast_{0} S)(v+v_{i})\in L(X)$ for all $v\in {\mathbb N}_{0}$ and $j\in {\mathbb N}_{n} \setminus M;$
\item[(c)] $[\sum_{i\in M}a_{i}(0)A_{i}]^{-1}A_{j}S(v)\in L(X)$ for all  $v\in {\mathbb N}_{0}$ and $j\in M.$
\end{itemize}
Then, for every $v\in {\mathbb N}$ and $x\in X$, we have
\begin{align}
\notag  & S\bigl(v+v_{max}\bigr)x=\Biggl[\sum_{i\in M}a_{i}(0)A_{i}\Biggr]^{-1}\Biggl\{ BS(v)x-k(v)Cx
\\\label{formulaq}&-\sum_{i\in {\mathbb N}_{n}\setminus M}A_{i}\bigl( a_{i}\ast_{0} S\bigr)(v+v_{i})x-\sum_{i\in M}A_{i}\Bigl[\bigl( a_{i}\ast_{0} S\bigr)(v+v_{i})x-a_{i}(0)S(v+v_{max})x\Bigr]\Biggr\},
\end{align}
and $A_{i}S(v)\in L(X)$ for all $ i\in M$ and $v\in {\mathbb N}_{0}$.
\item[(ii)] Suppose that $B$, $A_{1},...,A_{n}$ are closed linear operators on $X$, $C\in L(X)$, $k : {\mathbb N}_{0} \rightarrow {\mathbb C},$ $k(0)\neq 0,$ \eqref{kond} holds, and the operators $S(0)\in L(X),...,S(v_{max})\in L(X)$ satisfy \eqref{dokle}.
Suppose, further, that:
\begin{itemize}
\item[(a)] $[\sum_{i\in M}a_{i}(0)A_{i}]^{-1}\in L(X)$ and $[\sum_{i\in M}a_{i}(0)A_{i}]^{-1}B\in L(X);$
\item[(b)] $A_{j}\in L(X)$ for all $j\in {\mathbb N}_{n} \setminus M;$
\item[(c)] $[\sum_{i\in M}a_{i}(0)A_{i}]^{-1}A_{j}S(v)\in L(X)$ for all  $v\in {\mathbb N}_{v_{max}}^{0}$ and $j\in M.$
\end{itemize}
Define, for every $v\in {\mathbb N}$, the operator $S\bigl(v+v_{max})$ through \eqref{formulaq}.
Then $(S(v))_{v\in {\mathbb N}_{0}}\subseteq L(X)$ is a discrete
$(k,C,B,(A_{i})_{1\leq i\leq n},(v_{i})_{1\leq i\leq n})$-existence family and $A_{i}S(v)\in L(X)$ for all $ i\in M$ and $v\in {\mathbb N}_{0}$. Furthermore,
if
there exist a closed linear operator $A$ and the complex polynomials $P_{B}(\cdot),$ $ P_{1}(\cdot),...,\ P_{n}(\cdot)$ such that $CA\subseteq AC,$ 
$B=P_{B}(A),$ $A_{1}=P_{1}(A),...,\ A_{n}=P_{n}(A),$ and conditions \emph{(a)}-\emph{(c)} hold,
then $(S(v))_{v\in {\mathbb N}_{0}}$ is a discrete\\ $(k,C,B,(A_{i})_{1\leq i\leq n},\emptyset,(v_{i})_{1\leq i\leq n})$-existence family provided that\\ $S(v)A\subseteq AS(v)$ for $0\leq v\leq v_{max}.$
\item[(iii)] Suppose that the assumptions in \emph{(ii)} hold, $\sum_{v=0}^{+\infty}|a_{i}(v)|<+\infty$ for $1\leq i\leq n$, $\sum_{v=0}^{\infty}|k(v)|<+\infty,$ and \emph{(a)} or \emph{(b)} holds, where:
\begin{itemize}
\item[(a)] $A_{i}\in L(X)$ for all $i\in {\mathbb N}_{n},$ and
\begin{align*}
1>\Biggl\|\Biggl[\sum_{j\in M}a_{j}(0)A_{j}\Biggr]^{-1}B\Biggr\|&+\sum_{v=0}^{+\infty}\bigl| a_{i}(v)\bigr| \cdot  \sum_{i\in {\mathbb N}_{n}\setminus M} \Biggl\| \Biggl[\sum_{j\in M}a_{j}(0)A_{j}\Biggr]^{-1} A_{i}\Biggr\|
\\&+\sum_{v=1}^{+\infty}\bigl| a_{i}(v)\bigr| \cdot  \sum_{i\in M} \Biggl\|\Biggl[\sum_{j\in M}a_{j}(0)A_{j}\Biggr]^{-1}A_{i}\Biggr\|.
\end{align*}
\item[(b)] Suppose that there exist a closed linear operator $A$ and the complex polynomials $P_{B}(\cdot),$ $ P_{i}(\cdot)$ ($i\in M$) such that 
$B=P_{B}(A),$ $A_{i}=P_{i}(A)$ for $i\in M$, $A_{j}\in L(X)$ for all $j\in {\mathbb N}\setminus M,$ and
\begin{align*}
1>\Biggl\|\Biggl[\sum_{j\in M}a_{j}(0)A_{j}\Biggr]^{-1}B\Biggr\|&+\sum_{v=0}^{+\infty}\bigl| a_{i}(v)\bigr| \cdot  \sum_{i\in {\mathbb N}_{n}\setminus M} \Biggl\| \Biggl[\sum_{j\in M}a_{j}(0)A_{j}\Biggr]^{-1} A_{i}\Biggr\|
\\&+\sum_{v=1}^{+\infty}\bigl| a_{i}(v)\bigr| \cdot  \sum_{i\in M} \Biggl\|A_{i}\Biggl[\sum_{j\in M}a_{j}(0)A_{j}\Biggr]^{-1}\Biggr\|.
\end{align*}
\end{itemize}
Then we have
\begin{align}\label{ntc1}
\sum_{v=0}^{+\infty}\| S(v)\|<+\infty \ \ \mbox{ and }\ \
\sum_{v=0}^{+\infty}\Bigl\|A_{i}\bigl( a_{i}\ast_{0} S\bigr)\bigl(v+v_{i}\bigr)\Bigr\|<+\infty \ \ (1\leq i\leq n),
\end{align} 
provided that \emph{(a)} holds, resp.
we have \eqref{ntc1} and
\begin{align}\label{potrebno1}
\sum_{v=0}^{+\infty}\Bigl\| A_{i}S\bigl(v\bigr)\Bigr\|<+\infty \ \ (i\in M),
\end{align}
provided that \emph{(b)} holds and $A_{j}A_{i}\subseteq A_{i}A_{j}$ for every $i\in M$ and $j\in {\mathbb N}_{n}\setminus M.$
\end{itemize}
\end{thm}

\begin{proof}
It is clear that \eqref{dokle} holds (cf. \eqref{bnm} with $v=0$) as well as that Proposition \ref{prosto} implies $S(v)x\in D(A_{i})$ for all $v\in {\mathbb N}_{0},$ $x\in X$, $i\in M.$ Plugging $v=1$ in \eqref{bnm} and using this fact, we get
\begin{align*}
\Biggl[ \sum_{i\in M}a_{i}(0)A_{i}\Biggr]  S\bigl( v_{max}+1 \bigr)x&=BS(1)x-k(1)Cx
-\sum_{i\in {\mathbb N}_{n}\setminus M}A_{i}\bigl( a_{i}\ast_{0} S\bigr)(1+v_{i})x
\\& -\sum_{i\in M}A_{i}\Bigl[\bigl( a_{i}\ast_{0} S\bigr)(1+v_{i})x-a_{i}(0)S(1+v_{max})x\Bigr].
\end{align*}
If the operator $\sum_{i\in M}a_{i}(0)A_{i}$ is injective, the above implies that the value of $S( v_{max}+1)$ is uniquely determined; moreover, if (a)-(c) hold, then we obtain that the formula \eqref{formulaq} is valid with $v=1$ as well as that  $A_{i}S(v_{max}+1)\in L(X)$ for all $ i\in M.$ Repeating this argumetation with $v=2,$ $v=3,...,$ we obtain the part (i). The first conclusion in (ii) follows similarly, while the remaining parts can be deduced as in the proof of Theorem \ref{djeca}.
\end{proof}

It is also worth noting that, in Definition \ref{nsbgd}, we can consider arbitrary sequence $(h(v;x))_{v\in {\mathbb N}_{0}}$ in $X$ as a substitute of the already considered sequence $(k(v)Cx)_{v\in {\mathbb N}_{0}},$ especially in the case that $v_{max}>0.$

\subsection{The well-posedness and asymptotically almost periodic type solutions of the corresponding discrete abstract Cauchy problems}\label{dacp}

In this subsection, we will briefly consider the following discrete abstract Cauchy problem:
\begin{align}\label{sorry}
Bu(v)=f(v)+\sum_{i=1}^{n}A_{i}\bigl(a_{i}\ast_{0} u\bigr)(v+v_{i}),\quad v\in {\mathbb N}_{0},
\end{align}
where $v_{1},...,v_{n} \in {\mathbb N}_{0}.$
We say that a mapping $u: {\mathbb N}_{0}\rightarrow X$ is:
\begin{itemize}
\item[(i)] a mild solution of \eqref{sorry} if and only if $u(v)\in D(B)$ and $(a_{i}\ast_{0} u)(v+v_{i})\in D(A_{i})$ for all $v\in {\mathbb N}_{0},$ $i\in  {\mathbb N}_{n}$ and \eqref{dacp} holds;
\item[(ii)] a strong solution of \eqref{sorry} if and only if $u(v) \in D(A_{i})$ for all $v\in {\mathbb N}_{0},$ $i\in  {\mathbb N}_{n}$ and
\begin{align*}
Bu(v)=f(v)+\sum_{i=1}^{n}\bigl(a_{i}\ast_{0} A_{i}u\bigr)(v+v_{i}),\quad v\in {\mathbb N}_{0}.
\end{align*}
\end{itemize}

It is obvious that any strong solution of \eqref{sorry} is a mild solution of \eqref{sorry}.
If $(S(v))_{v\in {\mathbb N}_{0}}$ is a discrete $(k,C,B,(A_{i})_{1\leq i\leq n},(v_{i})_{1\leq i\leq n})$-existence family, then it is also clear that the function
$u(v):=S(v)x,$ $v\in {\mathbb N}_{0},$ where $x\in X,$ is a mild solution of \eqref{sorry}. Since we have assumed \eqref{kond}, it readily follows that
any mild solution of \eqref{sorry} is also a strong solution of \eqref{sorry}, if for each $i\in {\mathbb N}_{n}$ with $v_{i}>0$ we have 
$u(0)\in D(A_{i}),...,u(v_{i}-1)\in D(A_{i}).$ Furthermore, the problem \eqref{sorry} has at most one mild (strong) solution provided that
for each $i\in {\mathbb N}_{n}$ with $v_{i}>0$ we have 
$u(0)\in D(A_{i}),...,u(v_{i}-1)\in D(A_{i})$ as well as that the operator $B-\sum_{i=1}^{n}a_{i}(0)A_{i}
$ is injective; cf. also the proof of Proposition \ref{prosto}. The interested reader may try to prove some discrete variation of parameters formula for the problem \eqref{sorry}; cf. also \cite[Proposition 3.3]{avdi}.

In the remainder of this subsection, we will consider the case $v_{1}=v_{2}=...=v_{n}=0$, only. The proof of following result is simple and therefore omitted:

\begin{prop}\label{rezaqq}
Suppose that $B$, $A_{1},...,A_{n}$ are closed linear operators on $X$, $C\in L(X)$, $k : {\mathbb N}_{0} \rightarrow {\mathbb C},$ $k\neq 0$, $g : {\mathbb N}_{0} \rightarrow {\mathbb C},$ $g\neq 0$,
and $(S(v))_{v\in {\mathbb N}_{0}}\subseteq L(X)$ is a discrete 
$(k,C,B,(A_{i})_{1\leq i\leq n})$-existence family. If $x\in X,$
then 
the sequence $u(v):=(g\ast_{0}S)(v)x,$ $v\in {\mathbb N}_{0}$ is a strong solution of \eqref{sorry} with $f(v)=(k\ast_{0} g)(v),$ $v\in {\mathbb N}_{0}.$
\end{prop}

Now we will explain how Proposition \ref{rezaqq} can be applied in the analysis of the existence and uniqueness of asymptotically almost periodic type solutions of the discrete abstract Cauchy problem \eqref{sorry}, with $v_{1}=v_{2}=...=v_{n}=0;$ concerning the almost periodic functions (sequences), the almost automorphic functions (sequences) and their applications, one may refer, e.g., to the monographs \cite{diagana,nova-mono,nova-selected,metrical,188}. 

Suppose that there exist a bounded sequence $h : {\mathbb Z}\rightarrow X$ and a 
sequence $q : {\mathbb N}_{0}\rightarrow X$ vanishing at plus infinity such that $g(v)=h(v)+q(v)$ for all $v\in {\mathbb N}_{0}.$ If the requirements of Proposition \ref{rezaqq} hold, then we know that the function $u(v):=(g\ast _{0}S)(v)x,$ $x\in X,$ $ v\in {\mathbb N}_{0}$ is a strong solution of \eqref{sorry} with $f(v)=(k\ast_{0} g)(v),$ $v\in {\mathbb N}_{0}.$
If, additionally, $\sum_{v=0}^{+\infty} \| S(v)\|<+\infty,$
then the argumentation given at the end of \cite[Subsection 3.1]{avdi}
shows that there exist a bounded sequence $H : {\mathbb Z}\rightarrow X$ and a 
sequence $Q: {\mathbb N}_{0}\rightarrow X$ vanishing at plus infinity such that $u(v)=H(v)+Q(v)$ for all $v\in {\mathbb N}_{0};$ moreover, $H(\cdot)$ is
bounded and $T$-almost periodic, where $T\in L(X),$ or bounded and ${\mathrm R}$-almost automorphic, where ${\mathrm R}$ is any collection of sequences in ${\mathbb Z},$  provided that the forcing term $h(\cdot)$ has the same property (cf. \cite{avdi} for the notion).

\section{The abstract multi-term fractional difference equations with Weyl fractional derivatives}\label{rema-weyl}

Given a one-dimensional sequence $(u_{k})$ in $X$, the Euler forward difference operator $\Delta$ is defined by $\Delta u_{k}:=u_{k+1}-u_{k}.$ The operator $\Delta^{m}$ is defined inductively; then, for every integer $m\geq 1,$ we have:
\begin{align*}
\Delta^{m}u_{k}=\sum_{j=0}^{m}(-1)^{m-j}\binom{m}{j}u_{k+j}.
\end{align*}
If $u: {\mathbb N}_{0} \rightarrow X$ is a given sequence and $\alpha \in (0,\infty) \setminus {\mathbb N}_{0},$ then we
define the fractional integral $\Delta^{-\alpha}u : {\mathbb N}_{0} \rightarrow X$ by
$$
\Delta^{-\alpha}u(v):=\sum_{j=0}^{v}k^{\alpha}(v-j)u(j),\quad v\in {\mathbb N}_{0}.
$$
Set  $m:=\lceil \alpha \rceil.$ Then the Riemann-Liouville fractional difference operator of
order $\alpha>0,$ $\Delta^{\alpha}$ shortly, 
is defined by
$$
\Delta^{\alpha}u(v):=\Bigl[\Delta^{m}\bigl( \Delta^{-(m-\alpha)}u \bigr)\Bigr](v),\quad v\in {\mathbb N}_{0}.
$$

Further on,
if $\alpha>0,$ $m=\lceil \alpha \rceil$ and $u: {\mathbb Z}\rightarrow X$ satisfies $\sum_{v=-\infty}^{\infty}\| u(v)\| \cdot (1+|v|)^{m-\alpha-1}<+\infty,$ then we define the Weyl fractional derivative 
\begin{align*}
\Bigl[\Delta^{\alpha}_{W}u\Bigr](v):=\Bigl[\Delta^{m}\Bigl( \Delta^{-(m-\alpha)}_{W}u \Bigr)\Bigr](v),\quad v\in {\mathbb Z},
\end{align*}
where 
$$
\Bigl(\Delta^{-(m-\alpha)}_{W}u\Bigr)(v):=\sum_{l=-\infty}^{v}k^{m-\alpha}(v-l)u(l),\quad v\in {\mathbb Z}.
$$
Hence,
\begin{align}\label{fora}
\Bigl[\Delta^{\alpha}_{W}u\Bigr](v)=\sum_{j=0}^{m}\sum_{l=-\infty}^{v+j}(-1)^{m-j}\binom{m}{j}k^{m-\alpha}(v+j-l)u(l),\quad v\in {\mathbb Z}.
\end{align}
Due to \cite[Remark 2.4]{abadias}, the assumption $\sum_{v=-\infty}^{\infty}\| u(v)\| \cdot (1+|v|)^{m-\alpha-1}<+\infty$ implies 
\begin{align}\label{strah}
\Bigl[\Delta^{\alpha}_{W}u\Bigr](v)=\Bigl[\Delta^{m}\Bigl( \Delta^{-(m-\alpha)}_{W}u \Bigr)\Bigr](v)=\Bigl[\Delta^{-(m-\alpha)}_{W}\Bigl( \Delta^{m} u \Bigr)\Bigr](v),\quad v\in {\mathbb Z};
\end{align}
furthermore, it can be simply proved that, for every real number $a,$ the sequence $h(\cdot)=u(\cdot+a)$ satisfies $\sum_{v=-\infty}^{\infty}\| h(v)\| \cdot (1+|v|)^{m-\alpha-1}<+\infty$ as well as that (cf. \eqref{fora}): 
\begin{align}\label{nzl}
\Bigl[\Delta^{\alpha}_{W}h\Bigr](v)=\Bigl[\Delta^{\alpha}_{W}u\Bigr](v+a),\quad v\in {\mathbb Z}.
\end{align}
In the sequel, we also write $\Delta^{\alpha}_{W}u(v)$ in place of $[\Delta^{\alpha}_{W}u](v).$

If $A$ is a closed linear operator, then
we must distinguish the terms $A(\Delta^{\alpha}_{W}u)$ and $\Delta^{\alpha}_{W}(Au)$ for $\alpha>0.$ 
For example, if $u(v)\equiv x\notin D(A),$ then we have 
\begin{align}
\notag \Delta^{\alpha}_{W}u(v)&=\sum_{j=0}^{m}\sum_{l=-\infty}^{v+j}(-1)^{m-j}\binom{m}{j}k^{m-\alpha}(v+j-l) x
\\\label{iksa}&=\sum_{j=0}^{m}(-1)^{m-j}\binom{m}{j}\Biggl[\sum_{s=0}^{\infty}k^{m-\alpha}(s)x\Biggr]=0, \quad v\in {\mathbb Z},
\end{align}
but the term $[\Delta^{\alpha}_{W}(Au)](v)$ is not well-defined for any  $v\in {\mathbb Z}.$

Define now $T_{i,L}^{W}u(v):=A_i\Delta_W^{\alpha_i}u(v)$, if $v\in {\mathbb Z}$ and $i\in\mathbb N_{n}$, and $T_{i,R}^{W}u(v):=\Delta_C^{\alpha_i}A_iu(v)$, if $v\in {\mathbb Z}$ and $i\in\mathbb N_{n}$.
We assume that, for every $v\in {\mathbb N}_{0}$ and $i\in\mathbb N_{n}$, $T_{i}^{W}u(v)$ denotes either $T_{i,L}^{W}u(v)$ or $T_{i,R}^{W}u(v)$. 
Of concern is the following abstract multi-term fractional difference equation without initial conditions:
\begin{equation}\label{DFDEabzw}
\sum_{i=1}^{n}T_{i}^{W}u(v+v_{i})=f(v),\quad v\in {\mathbb Z},
\end{equation}
where $v_1,\dots,v_{n} \in {\mathbb Z}$ and $f: {\mathbb Z} \rightarrow X$ is a given sequence. A
strong solution of problem \eqref{DFDEabzw}  is any sequence $u: {\mathbb Z}\rightarrow X$ such that the term $T_{i}^{W}u(\cdot+v_{i})$ is well-defined for any $i\in {\mathbb N}_{n}$ as well as that \eqref{DFDEabzw} identically holds for any $v\in {\mathbb Z}$.

Set now $v_{min}:=\min\{v_{1},...,v_{n}\}$ and $y(v):=u(v+v_{min}),$ $v\in {\mathbb Z}.$ Then it is clear that the problem \eqref{DFDEabzw}
is equivalent with the probem (cf. \eqref{nzl}):
\begin{equation*}
\sum_{i=1}^{n}T_{i}^{W}y\bigl(v+[v_{i}-v_{min}]\bigr)=f(v),\quad v\in {\mathbb Z},
\end{equation*}
as well as that $v_{i}-v_{min}\geq 0$ for all $i\in {\mathbb N}_{n}.$ Therefore, in the analysis of the existence and uniqueness of solutions of the problem \eqref{DFDEabzw}, 
we may assume without loss of generality that $v_{1}\geq 0,\ ...,\ v_{n}\geq 0$. This will be our standing assumption henceforth.

Due to \eqref{fora}, the problem \eqref{DFDEabzw}
is equivalent with 
\begin{align*}
&\sum_{i\in {\mathcal I}}A_{i}\Biggl[\sum_{j=0}^{m_{i}}\sum_{l=-\infty}^{v+j}(-1)^{m_{i}-j}\binom{m_{i}}{j}k^{m_{i}-\alpha}(v+j-l)u(l)\Biggr]
\\& +\sum_{i\notin {\mathcal I}}\Biggl[\sum_{j=0}^{m_{i}}\sum_{l=-\infty}^{v+j}(-1)^{m_{i}-j}\binom{m_{i}}{j}k^{m_{i}-\alpha}(v+j-l)A_{i}u(l)\Biggr]=f(v),
\quad v\in {\mathbb Z},
\end{align*}
where ${\mathcal I}:=\{i\in\mathbb N_{n}:\alpha_i>0\text{ and }T_{i,L}^{W}u(v)\text{ appears on the left hand side of }\eqref{DFDEabzw}\}$ 
We cannot expect the uniqueness of (periodic/almost periodic) strong solutions of \eqref{DFDEabzw} in any sense: if $u(\cdot)$ is a strong solution of \eqref{DFDEabzw}, then the sequence $u_{x}(\cdot):=u(\cdot)+x,$ where $x\in D(A_{i})$ for all $ i\notin {\mathcal I}$, is likewise a strong solution of \eqref{DFDEabzw}; cf. \eqref{iksa}.

We continue our exposition with the observation that \cite[Theorem 3.4]{avdi} can be helpful in the analysis of the existence of strong solutions to some classes of the abstract multi-term fractional difference equations. This result is far from being satisfactory in the general analysis of problem \eqref{DFDEabzw}, when we need the notion introduced in Definition \ref{nsbgd}.

The following result seems to be extremely important in the study of the abstract multi-term fractional difference equations with Weyl fractional derivatives: 

\begin{thm}\label{zajeb}
Suppose that $v_{1}\geq 0,\ ...,\ v_{n}\geq 0$, $(S(v))_{v\in {\mathbb N}_{0}}\subseteq L(X)$ is a discrete $(k,C,B,(A_{i})_{1\leq i\leq n},(v_{i})_{1\leq i\leq n})$-existence family, 
\eqref{ntc}
and the following hold:
\begin{itemize}
\item[(a)] $f : {\mathbb Z}\rightarrow X$ is a bounded sequence, $k\in l^{1}({\mathbb N}_{0} :Y)$  and $\sum_{v=0}^{+\infty}| a_{i}(v)|<+\infty$ for $1\leq i\leq n,$ or
\item[(b)] $f\in l^{1}({\mathbb Z} : X),$ $k : {\mathbb N}_{0}\rightarrow X$ is a bounded sequence and $a_{i} : {\mathbb Z}\rightarrow {\mathbb C}$ is a bounded sequence for $1\leq i\leq n.$
\end{itemize}
Define 
$$
u(v):=\sum_{l=-\infty}^{v}S(v-l)f(l),\quad v\in {\mathbb Z}
$$
and
\begin{align}
\notag g(v)&:=A_{1}\sum_{l=v+1}^{v+v_{1}}\bigl( a_{1}\ast_{0} S\bigr)(v+v_{1}-l)f(l)+...
\\\label{getri}&+A_{n}\sum_{l=v+1}^{v+v_{n}}\bigl( a_{n}\ast_{0} S\bigr)(v+v_{n}-l)f(l),\quad v\in {\mathbb Z}.
\end{align}
Then $u(\cdot)$ is bounded if \emph{(a)} holds, $u\in l^{1}({\mathbb Z} : X)$ if \emph{(b)} holds, and we have
\begin{align}
\notag Bu(v)&=A_{1}\sum_{l=-\infty}^{v+v_{1}}a_{1}(v+v_{1}-l)u(l)+...
\\& \label{bnm1}+A_{n}\sum_{l=-\infty}^{v+v_{n}}a_{n}(v+v_{n}-l)u(l)+(k \circ Cf)(v)+g(v),\ v\in {\mathbb Z}.
\end{align}
Especially, the following holds:
\begin{itemize}
\item[(i)] Suppose that $f\in l^{1}({\mathbb Z} : X)$ and $g(v)$ is given by \eqref{getri}, with $a_{i}(v)=k^{m_{i}-\alpha_{i}}(v)$ for all $v\in {\mathbb N}_{0}$ and $i\in {\mathbb N},$ where $m=m_{i}\in {\mathbb N}$ for all $i\in {\mathbb N}_{n}.$ Then we have:
\begin{align*}
\bigl( \Delta^{m} Bu\bigr)(v)&=A_{1}\Bigl(\Delta^{\alpha_{1}}_{W}u\Bigr)(v+v_{1})+...
\\& +A_{n}\Bigl(\Delta^{\alpha_{n}}_{W}u\Bigr)(v+v_{n})+\Delta^{m}(k \circ Cf)(v)+\Delta^{m}g(v),\ v\in {\mathbb Z}.
\end{align*}
\item[(ii)] Suppose that $f\in l^{1}({\mathbb Z} : X)$ and $u=\Delta^{m_{n}}h$ for a certain sequence $h : {\mathbb Z} \rightarrow X.$ Then we have:
\begin{align*}
B\bigl( \Delta^{m_{n}} h\bigr)(v)&=\sum_{j=1}^{n-1}A_{j}\Bigl(\Delta^{m_{n}-m_{j}}\Delta^{\alpha_{j}}_{W}h\Bigr)(v+v_{j})
\\&+A_{n}\Bigl(\Delta^{\alpha_{n}}_{W}h\Bigr)(v+v_{n})+(k \circ Cf)(v)+g(v),\ v\in {\mathbb Z}.
\end{align*}
\end{itemize}
\end{thm}

\begin{proof}
It is clear that $u(\cdot)$ is well-defined and bounded if (a) holds as well as $u\in l^{1}({\mathbb Z} : X)$ if (b) holds; furthermore, the sequence $(k \circ Cf)(\cdot) $ is well-defined. The functional equality \eqref{bnm} and the closedness of the linear operator $B$ together 
imply with the prescribed assumptions and the Fubini theorem that:{\small
\begin{align} 
\notag & Bu(v)-(k \circ Cf)(v)=\sum_{l=-\infty}^{v}BS(v-l)f(l)-(k \circ Cf)(v)=-(k \circ Cf)(v)
\\\notag & +\sum_{l=-\infty}^{v}\Biggl[ k(v-l)Cf(l)+A_{1}\bigl( a_{1}\ast_{0} S\bigr)(v+v_{1}-l)f(l)+...+A_{n}\bigl( a_{n}\ast_{0} S\bigr)(v+v_{n}-l)f(l) \Biggr]
\\\notag & =\sum_{l=-\infty}^{v}\Biggl[ A_{1}\bigl( a_{1}\ast_{0} S\bigr)(v+v_{1}-l)f(l)+...+A_{n}\bigl( a_{n}\ast_{0} S\bigr)(v+v_{n}-l)f(l) \Biggr]
\\\label{pojasni} & =A_{1}\sum_{l=-\infty}^{v}\bigl( a_{1}\ast_{0} S\bigr)(v+v_{1}-l)f(l)+...+A_{n}\sum_{l=-\infty}^{v}\bigl( a_{n}\ast_{0} S\bigr)(v+v_{n}-l)f(l)
\\\notag& =A_{1}\sum_{l=-\infty}^{v+v_{1}}\bigl( a_{1}\ast_{0} S\bigr)(v+v_{1}-l)f(l)+...+A_{n}\sum_{l=-\infty}^{v+v_{n}}\bigl( a_{n}\ast_{0} S\bigr)(v+v_{n}-l)f(l)-g(v)
\\\notag & =\sum_{j=1}^{n} A_{j}\sum_{s=0}^{+\infty}\Biggl[ a_{j}(s)S(0)+...+a_{j}(0)S(s)\Biggr]f(v+v_{j}-s)-g(v)
\\\label{otac}  & =\sum_{j=1}^{n}A_{j}\sum_{s=0}^{+\infty}a_{j}(s)\sum_{r=0}^{+\infty}S(r)f(v+v_{j}-s-r)-g(v)
\\\notag & =\sum_{j=1}^{n}A_{j}\sum_{l=-\infty}^{v+v_{j}}a_{j}(v+v_{j}-l)u(l)-g(v),\quad v\in {\mathbb Z},
\end{align}}
so that \eqref{bnm1} holds true; here, \eqref{pojasni} follows from the assumption
 $\sum_{v=0}^{+\infty}\|A_{i}\bigl( a_{i}\ast_{0} S\bigr)(v+v_{i})\|<+\infty$ for $1\leq i\leq n$ 
and the closedness of the linear operators $A_{1},...,A_{n}$, while
\eqref{otac} follows from an application of the discrete Fubini theorem, the assumption $\sum_{v=0}^{+\infty}\| S(v)\|<+\infty$ and the assumption (a) or (b).
The proofs of (i) and (ii) follow simply from this general result with condition (b) being satisfied, the definitions of Weyl fractional derivatives and the equation \eqref{strah}.
\end{proof}

We continue be providing certain observations:

\begin{rem}\label{finite-differences}
\begin{itemize}
\item[(i)] It is clear that $g(\cdot) \equiv 0,$ if $v_{1}=v_{2}=...=v_{n}=0$.
\item[(ii)] The assumption $u(v)=\Delta^{m_{n}}h(v),$ $v\in {\mathbb Z}$ in (ii) implies that $h(v)=u_{1}(v)+ c_{0}+c_{1}v+...+c_{m_{n}-1}v^{m_{n}-1}$ for all $v\in {\mathbb Z}$, for some sequence $(u_{1}(\cdot))$ depending on $(u(\cdot))$ and arbitrary complex constants $c_{0},...,c_{m_{n}-1}.$ For more details about the calculus of finite differences and its applications, we refer the reader to the monographs \cite{fd1} and \cite{fd2}.
\item[(iii)] Suppose that $c_{i}\in {\mathbb C}$, $v_{1}=v_{2}=...=v_{n}=0$ and $h_{i}\in {\mathbb Z}$ for $1\leq i\leq m.$ Define $y(v):=\sum_{i=1}^{m}c_{i}u(v+h_{i}),$ $v\in {\mathbb Z};$ then \eqref{bnm1} implies after the substitution $s=v-l\geq 0$ that 
\begin{align*}
By(v)&=A_{1}\sum_{l=-\infty}^{v}a_{1}(v-l)y(l)+...+A_{n}\sum_{l=-\infty}^{v}a_{n}(v-l)y(l)
\\&+\sum_{i=i}^{m}c_{i}\bigl(k \circ Cf\bigr)\bigl(v +h_{i}\bigr),\ v\in {\mathbb Z}.
\end{align*}
\end{itemize}
\end{rem}

We can similarly prove the following general result (the
strong solutions of problems under our consideration are obtained by plugging ${\mathcal I}=\emptyset$):

\begin{thm}\label{zajebq}
Suppose that $v_{1}\geq 0,\ ...,\ v_{n}\geq 0$, ${\mathcal I}\subseteq {\mathbb N}_{n},$ $(S(v))_{v\in {\mathbb N}_{0}}\subseteq L(X)$ is a discrete $(k,C,B,(A_{i})_{1\leq i\leq n},(v_{i})_{1\leq i\leq n},{\mathcal I})$-existence family, 
$\sum_{v=0}^{+\infty}\| S(v)\|<+\infty$, $\sum_{v=0}^{+\infty}\|A_{i}\bigl( a_{i}\ast_{0} S\bigr)(v+v_{i})\|<+\infty$ for $i\in {\mathcal I}$ and the following holds:
\begin{itemize}
\item[(a)] $f : {\mathbb Z}\rightarrow X$ is a bounded sequence, $k\in l^{1}({\mathbb N}_{0} : X)$ and $\sum_{v=0}^{+\infty}| a_{i}(v)|<+\infty$ for $i\in {\mathcal I}$ or
\item[(b)] $f\in l^{1}({\mathbb Z} : X),$ $k : {\mathbb N}_{0} \rightarrow X$ is a bounded sequence and $a_{i} : {\mathbb Z}\rightarrow {\mathbb C}\setminus \{0\}$ is a bounded sequence for $i\in {\mathcal I}$
\end{itemize}
as well as
\begin{itemize}
\item[(c)] $A_{i}f : {\mathbb Z}\rightarrow X$ is a bounded sequence, $\sum_{v=0}^{+\infty}| a_{i}(v)|<+\infty$ for $i\in {\mathbb N}_{n}\setminus {\mathcal I}$ and $(S(v))_{v\in {\mathbb N}_{0}}\subseteq L(X)$ is a discrete $(k,C,B,(A_{i})_{1\leq i\leq n},(v_{i})_{1\leq i\leq n}, {\mathcal I})$-existence family, or
\item[(d)] $f\in l^{1}({\mathbb Z} : X)$, $\sum_{v=0}^{+\infty}\|A_{i}S(v)\|<+\infty$ for all $i\in {\mathbb N}_{n}\setminus {\mathcal I}$ and $a_{i} : {\mathbb N}_{0}\rightarrow {\mathbb C}\setminus\{0\}$ is a bounded sequence for $i\in {\mathbb N}_{n}\setminus {\mathcal I},$ or
\item[(e)] $f\in l^{1}({\mathbb Z} : [D(A_{i})])$ for all $i\in {\mathbb N}_{n}\setminus {\mathcal I},$ $a_{i} : {\mathbb N}_{0}\rightarrow {\mathbb C}\setminus\{0\}$ is a bounded sequence for $i\in {\mathbb N}_{n}\setminus {\mathcal I}$ and $(S(v))_{v\in {\mathbb N}_{0}}\subseteq L(X)$ is a discrete\\ $(k,C,B,(A_{i})_{1\leq i\leq n},(v_{i})_{1\leq i\leq n}, {\mathcal I})$-existence family.
\end{itemize}
Define 
$
u(\cdot) 
$ and $g(\cdot)$ in the same way as in the proof of \emph{Theorem \ref{zajeb}}. 
Then we have:
\begin{align*}
Bu(v)&=\sum_{i\in {\mathcal I}}A_{i}\sum_{l=-\infty}^{v+v_{i}}a_{i}(v+v_{i}-l)u(l)
\\&+\sum_{i\in {\mathbb N}_{n}\setminus {\mathcal I}}\sum_{l=-\infty}^{v+v_{i}}a_{i}(v+v_{i}-l)A_{i}u(l)+(k \circ Cf)(v)+g(v),\ v\in {\mathbb Z}.
\end{align*}
Especially, the following holds for a function $f\in l^{1}({\mathbb Z} : X):$ 
\begin{itemize}
\item[(i)] Suppose that $a_{i}(v)=k^{m_{i}-\alpha_{i}}(v)$ for all $v\in {\mathbb N}_{0}$ and $i\in {\mathbb N},$
where $m=m_{i}\in {\mathbb N}$ for all $i\in {\mathbb N}_{n}.$ If $\sum_{v=0}^{+\infty}\|A_{i}S(v)\|<+\infty$ for all $i\in {\mathbb N}_{n}\setminus {\mathcal I}$ or $f\in l^{1}({\mathbb Z} : [D(A_{i})])$ for all $i\in {\mathbb N}_{n}\setminus {\mathcal I}$ and $(S(v))_{v\in {\mathbb N}_{0}}\subseteq L(X)$ is a discrete $(k,C,B,(A_{i})_{1\leq i\leq n},(v_{i})_{1\leq i\leq n}, {\mathcal I})$-existence family, then we have:
\begin{align*}
\bigl( \Delta^{m} Bu\bigr)(v)&=\sum_{i\in {\mathcal I}}A_{i}\Bigl(\Delta^{\alpha_{i}}_{W}u\Bigr)(v+v_{i})
\\&+\sum_{i\in{\mathbb N}_{n}\setminus {\mathcal I}}\Bigl(\Delta^{\alpha_{i}}_{W}A_{i}u\Bigr)(v+v_{i})+\Delta^{m} (k \circ Cf)(v)+\Delta^{m}  g(v),\quad v\in {\mathbb Z}.
\end{align*}
\item[(ii)] Suppose that $\sum_{v=0}^{+\infty}\|A_{i}S(v)\|<+\infty$ for all $i\in {\mathbb N}_{n}\setminus {\mathcal I}$ or $f\in l^{1}({\mathbb Z} : [D(A_{i})])$ for all $i\in {\mathbb N}_{n}\setminus {\mathcal I}$ and $(S(v))_{v\in {\mathbb N}_{0}}\subseteq L(X)$ is a discrete\\ $(k,C,B,(A_{i})_{1\leq i\leq n},(v_{i})_{1\leq i\leq n}, {\mathcal I})$-existence family as well as $u=\Delta^{m_{n}}h$ for a certain sequence $h : {\mathbb Z} \rightarrow \bigcap_{i\in {\mathbb N}_{n}\setminus {\mathcal I}}D(A_{i}).$ Then we have:
\begin{align*}
B\bigl( \Delta^{m_{n}} h\bigr)(v)&=\sum_{i\in {\mathcal I}}A_{i}\Bigl(\Delta^{m_{n}-m_{i}}\Delta^{\alpha_{i}}_{W}h\Bigr)(v+v_{i})
\\&+\sum_{i\in{\mathbb N}_{n}\setminus {\mathcal I}}\Bigl(\Delta^{m_{n}-m_{i}}\Delta^{\alpha_{j}}_{W}A_{i}h\Bigr)(v+v_{i})+(k \circ Cf)(v)+g(v),\ v\in {\mathbb Z}.
\end{align*}
\end{itemize}
\end{thm}

It is clear that Theorem \ref{zajeb} and Theorem \ref{zajebq} can be illustrated with many concrete examples in which the corresponding discrete $(k,0,0,(A_{i})_{1\leq i\leq n},(v_{i})_{1\leq i\leq n},{\mathcal I})$-existence families are integrable on account of Theorem \ref{djeca}(iii) or Theorem \ref{djeca-vmax}(iii).
Now we will provide one more important application of Theorem \ref{zajebq}:

\begin{example}\label{kucanje}
Suppose that $(T(t))_{t\geq 0}\subseteq L(X)$ is a strongly continuous operator family such that $\int^{+\infty}_{0}\| T(t)\|\, dt<+\infty$ and, for every $x\in X,$ the mapping $t\mapsto T(t)x,$ $t\geq 0$ is a mild solution of the abstract Cauchy problem
$$
A_{n}D_{t}^{\alpha_{n}}T(t)x+A_{n-1}D_{t}^{\alpha_{n-1}}T(t)x+...+A_{1}D_{t}^{\alpha_{1}}T(t)x=0,\quad t>0,
$$
accompanied with certain initial conditions (cf. \cite[pp. 161--163]{FKP} for more details in this direction),
where $0\leq \alpha_{1}<\alpha_{2}<...<\alpha_{n}$ and the operators $A_{j}$ are polynomials of a certain closed linear operator $A$;
the interested reader may consult \cite{knjigaho}-\cite{FKP} for more details and many interesting examples of differential operators which can be used to provide certain applications of the conclusions established here (we can also consider the exponentially bounded operator families $(T(t))_{t\geq 0}$ using the Poisson like transform from Theorem \ref{stru} since the statement of \cite[Theorem 5.5]{abadias1} admits an extension in this framework; details can be left to the interested readers).
Set 
$$
S(v)x:=\int^{+\infty}_{0}e^{-t}(t^{v}/v!)T(t)x\, dt, \quad v\in {\mathbb N}_{0},\ x\in X.
$$
Then we have $\sum_{v=0}^{+\infty}\| S(v)\|<+\infty$ and
\begin{align*}
A_{n}\bigl( \Delta^{\alpha_{n}}S(\cdot)x \bigr)(v)&+\sum_{s=1}^{n-1}A_{s}\bigl( \Delta^{\alpha_{s}}S(\cdot)x \bigr)(v+m_{n}-m_{s})
=0,\ v\in {\mathbb N}_{0},\ x\in X;
\end{align*}
cf. \cite{multi-term}. In other words,
\begin{align*}
 &A_{n}\Biggl[ \sum_{j=0}^{m_{n}}\sum_{l=0}^{v+j}(-1)^{m_{n}-j}\binom{m_{n}}{j}k^{m_{n}-\alpha_{n}}(v+j-l)S(l)x \Biggr]
\\&+\sum_{s=1}^{n-1}A_{s}\Biggl[ \sum_{j=0}^{m_{s}}\sum_{l=0}^{v+m_{n}-m_{s}+j}(-1)^{m_{s}-j}\binom{m_{s}}{j}k^{m_{s}-\alpha_{s}}(v+m_{n}-m_{s}+j-l)S(l)x\Biggr]=0,
\end{align*}
for any $v\in {\mathbb N}_{0}$ and $x\in X.$ Now it can be easily shown that Theorem \ref{zajebq} is applicable with ${\mathcal I}=\emptyset,$ $B=C=0$ and arbitrary non-trivial sequence $f\in l^{1}({\mathbb Z} : [D(A_{i})])$ for all $i\in {\mathbb N}_{n};$ in our concrete situation, we have that $(S(v))_{v\in {\mathbb N}_{0}}\subseteq L(X)$ is a discrete $(k,0,0,(A_{i})_{1\leq i\leq n},(v_{i})_{1\leq i\leq n},{\mathcal I})$-existence family with $v_{i}=m_{i}$ for
$1\leq i\leq n$ and 
$$
a_{s}(v):=\sum_{j=0}^{m_{s}}(-1)^{m_{s}-j}\binom{m_{s}}{j}k^{m_{s}-\alpha_{s}}\bigl(v+m_{n}+j-2m_{s}\bigr),\quad v\in  {\mathbb N}_{0} \ \ (1\leq s\leq n),
$$
where we have put $k^{m_{s}-\alpha_{s}}(v):=0$ if $v>0$ ($1\leq s\leq n$). Next, let us observe that we have $a_{s}(0)=1,$ $i\in M$ and $\sum_{v=0}^{+\infty}| a_{s}(v)|<+\infty$ for $s\in {\mathbb N}_{n}$ since $a_{s}(v)\sim c_{s}v^{-2+m_{s}-\alpha_{s}}$ as $v\rightarrow +\infty;$ this follows form a relatively simple analysis involving the identity $\sum_{j=0}^{m_{s}}(-1)^{m_{s}-j}\binom{m_{s}}{j}=0$, the Lagrange mean value theorem and the asymptotic formula $k^{\alpha}(v)\sim g_{\alpha}(v)$ as $v\rightarrow +\infty$ ($\alpha>0$). Consequently, we have that there do not exist an integer $m_{s}\in {\mathbb N},$ the complex constants $\beta_{1;s},...,\beta_{m;s}$ and the translation numbers $v_{1;s}'\in {\mathbb N}_{0},...,v_{m;s}'\in {\mathbb N}_{0}$ such that $k^{m_{s}-\alpha_{s}}(v)=\beta_{1;s}a_{s}(v+v_{1;s}')+...+\beta_{m;s}a_{s}(v+v_{m;s}').$

In our concrete situation, we are solving the equation
\begin{align}\label{je}
\sum_{l=-\infty}^{v+v_{n}}a_{n}(v+v_{n}-l)A_{n}u(l)+...+\sum_{l=-\infty}^{v+v_{1}}a_{1}(v+v_{1}-l)A_{1}u(l)=-g(v),\ v\in {\mathbb Z},
\end{align}
where $g(\cdot)$ is given by \eqref{getri}. Keeping in mind the representation of $a_{s}(\cdot)$ for $1\leq s\leq n,$ we get
\begin{align*}
& \sum_{l=-\infty}^{v+m_{n}}\sum_{j=0}^{m_{n}}(-1)^{m_{n}-j}\binom{m_{n}}{j}k^{m_{n}-\alpha_{n}}\bigl(v+j-l\bigr)A_{n}u(l)
\\&+\sum_{s=1}^{n-1}\sum_{l=-\infty}^{v+m_{s}}\sum_{j=0}^{m_{s}}(-1)^{m_{s}-j}\binom{m_{s}}{j}k^{m_{s}-\alpha_{s}}\bigl(v+m_{n}-m_{s}+j-l\bigr)A_{s}u(l)=-g(v),
\end{align*}
for any $v\in {\mathbb Z},$ i.e.,
 \begin{align*}
& \sum_{j=0}^{m_{n}}(-1)^{m_{n}-j}\binom{m_{n}}{j}\sum_{l=-\infty}^{v+m_{n}}k^{m_{n}-\alpha_{n}}\bigl(v+j-l\bigr)A_{n}u(l)
\\&+\sum_{s=1}^{n-1}\sum_{j=0}^{m_{s}}(-1)^{m_{s}-j}\binom{m_{s}}{j}\sum_{l=-\infty}^{v+m_{s}}k^{m_{s}-\alpha_{s}}\bigl(v+m_{n}-m_{s}+j-l\bigr)A_{s}u(l)=-g(v),
\end{align*}
for any $v\in {\mathbb Z}.$ Taking into account the formula \eqref{fora}, we finally get
\begin{align*}
\Bigl[\Delta^{\alpha_{n}}_{W}A_{n}u\Bigr](v)&+\Bigl[\Delta^{\alpha_{n-1}}_{W}A_{n-1}u\Bigr]\bigl(v+m_{n}-m_{n-1}\bigr)
\\&+...+\Bigl[\Delta^{\alpha_{1}}_{W}A_{1}u\Bigr]\bigl(v+m_{n}-m_{1}\bigr)=-g(v),\quad v\in {\mathbb Z}.
\end{align*}
The term $g(\cdot)$ can be directly computed for small values of $\alpha_{j}$ ($1\leq j\leq n$); for example, if $\alpha_{n}\leq 1,$ then we have
$$
g(v)=\sum_{i\in M}A_{i}\int_{0}^{+\infty}e^{-t}T(t)f(v+1)\, dt,\quad v\in {\mathbb Z}.
$$
\end{example}

The kernels $a(v)=k^{m-\alpha}(v),$ $v\in {\mathbb N}_{0}$ are not integrable provided that $\alpha \in (0,\infty) \setminus {\mathbb N}$ and $m=\lceil \alpha \rceil.$ Therefore, we cannot use Theorem \ref{djeca}(iii) or Theorem \ref{djeca-vmax}(iii) to deduce the integrability of the corresponding discrete $(k,B,C,(A_{i})_{1\leq i\leq n},(v_{i})_{1\leq i\leq n},{\mathcal I})$-existence families. We can overcome this difficulty by applying the following trick:

\begin{thm}\label{zajebqe}
Suppose that $v_{1}\geq 0,\ ...,\ v_{n}\geq 0$, ${\mathcal I}\subseteq {\mathbb N}_{n},$ $(S(v))_{v\in {\mathbb N}_{0}}\subseteq L(X)$ is a discrete $(k,C,B,(A_{i})_{1\leq i\leq n},(v_{i})_{1\leq i\leq n},{\mathcal I})$-existence family, $\omega>0,$\\
$\sum_{v=0}^{+\infty}\| e^{-\omega v}S(v)\|<+\infty$, $\sum_{v=0}^{+\infty}\|A_{i}\bigl( a_{i}\ast_{0} e^{-\omega \cdot}S(\cdot)\bigr)(v+v_{i})\|<+\infty$ for $i\in {\mathcal I}$ and the following holds:
\begin{itemize}
\item[(a)] $e^{-\omega \cdot}f : {\mathbb Z}\rightarrow X$ is a bounded sequence, $e^{-\omega \cdot}k\in l^{1}({\mathbb N}_{0} : X)$ and\\ $\sum_{v=0}^{+\infty}| a_{i}(v)e^{-\omega (\cdot-v_{i})}|<+\infty$ for $i\in {\mathcal I}$ or
\item[(b)] $e^{-\omega \cdot}f\in l^{1}({\mathbb Z} : X),$ $e^{-\omega \cdot}k : {\mathbb N}_{0} \rightarrow X$ is a bounded sequence and $e^{-\omega (\cdot-v_{i})}a_{i} : {\mathbb Z}\rightarrow {\mathbb C}\setminus \{0\}$ is a bounded sequence for $i\in {\mathcal I}$
\end{itemize}
as well as
\begin{itemize}
\item[(c)] $e^{-\omega \cdot}A_{i}f : {\mathbb Z}\rightarrow X$ is a bounded sequence, $\sum_{v=0}^{+\infty}| a_{i}(v)e^{-\omega (\cdot-v_{i})}|<+\infty$ for $i\in {\mathbb N}_{n}\setminus {\mathcal I}$ and and $(S(v))_{v\in {\mathbb N}_{0}}\subseteq L(X)$ is a discrete\\ $(k,C,B,(A_{i})_{1\leq i\leq n},(v_{i})_{1\leq i\leq n}, {\mathcal I})$-existence family, or
\item[(d)] $e^{-\omega \cdot}f\in l^{1}({\mathbb Z} : X)$, $\sum_{v=0}^{+\infty}\|A_{i}e^{-\omega v}S(v)\|<+\infty$ for all $i\in {\mathbb N}_{n}\setminus {\mathcal I}$ and $e^{-\omega (\cdot-v_{i})}a_{i} : {\mathbb N}_{0}\rightarrow {\mathbb C}\setminus\{0\}$ is a bounded sequence for $i\in {\mathbb N}_{n}\setminus {\mathcal I},$ or
\item[(e)] $e^{-\omega \cdot}f\in l^{1}({\mathbb Z} : [D(A_{i})])$ for all $i\in {\mathbb N}_{n}\setminus {\mathcal I},$ $e^{-\omega (\cdot-v_{i})}a_{i} : {\mathbb N}_{0}\rightarrow {\mathbb C}\setminus\{0\}$ is a bounded sequence for $i\in {\mathbb N}_{n}\setminus {\mathcal I}$ and $(S(v))_{v\in {\mathbb N}_{0}}\subseteq L(X)$ is a discrete $(k,C,B,(A_{i})_{1\leq i\leq n},(v_{i})_{1\leq i\leq n}, {\mathcal I})$-existence family.
\end{itemize}
Define 
$$
u(v):=e^{\omega v}\sum_{l=-\infty}^{v}\Bigl[e^{-\omega (v-l)}S(v-l)\Bigr] \cdot \Bigl[ e^{-\omega l}f(l)\Bigr],\quad v\in {\mathbb Z}
$$ and $g_{\omega}(\cdot)$ in the same way as in the proof of \emph{Theorem \ref{zajeb}} with the terms $g(\cdot),$ $a_{i}(\cdot),$ $S(\cdot)$ and $f(\cdot)$ replaced therein with the terms $g_{\omega}(\cdot)$, $e^{-\omega (\cdot-v_{i})}a_{i}(\cdot),$ $e^{-\omega \cdot}S(\cdot)$ and $e^{-\omega \cdot}f(\cdot)$, respectively.
Then we have:
\begin{align}
\notag Bu(v)&=\sum_{i\in {\mathcal I}}A_{i}\sum_{l=-\infty}^{v+v_{i}}a_{i}(v+v_{i}-l)u(l)
\\\label{got}&+\sum_{i\in {\mathbb N}_{n}\setminus {\mathcal I}}\sum_{l=-\infty}^{v+v_{i}}a_{i}(v+v_{i}-l)A_{i}u(l)+e^{-\omega v}(k \circ Cf)(v)+g_{\omega}(v),\ v\in {\mathbb Z}.
\end{align}
Especially, the following holds for a function $e^{-\omega \cdot}f\in l^{1}({\mathbb Z} : X):$ 
\begin{itemize}
\item[(i)] Suppose that $a_{i}(v)=k^{m_{i}-\alpha_{i}}(v)$ for all $v\in {\mathbb N}_{0}$ and $i\in {\mathbb N},$
where $m=m_{i}\in {\mathbb N}$ for all $i\in {\mathbb N}_{n}.$ If $\sum_{v=0}^{+\infty}\|A_{i}e^{-\omega v}S(v)\|<+\infty$ for all $i\in {\mathbb N}_{n}\setminus {\mathcal I}$ or $e^{-\omega \cdot}f\in l^{1}({\mathbb Z} : [D(A_{i})])$ for all $i\in {\mathbb N}_{n}\setminus {\mathcal I}$ and $(S(v))_{v\in {\mathbb N}_{0}}\subseteq L(X)$ is a discrete $(k,C,B,(A_{i})_{1\leq i\leq n},(v_{i})_{1\leq i\leq n},{\mathbb N}_{n}\setminus {\mathcal I})$-existence family, then we have:
\begin{align*}
\bigl( \Delta^{m} Bu\bigr)(v)&=\sum_{i\in {\mathcal I}}A_{i}\Bigl(\Delta^{\alpha_{i}}_{W}u\Bigr)(v+v_{i})
\\&+\sum_{i\in{\mathbb N}_{n}\setminus {\mathcal I}}\Bigl(\Delta^{\alpha_{i}}_{W}A_{i}u\Bigr)(v+v_{i})+\Delta^{m} e^{-\omega v}(k \circ Cf)(v)+\Delta^{m}  g_{\omega}(v),\quad v\in {\mathbb Z}.
\end{align*}
\item[(ii)] Suppose that $\sum_{v=0}^{+\infty}\|A_{i}e^{-\omega v}S(v)\|<+\infty$ for all $i\in {\mathbb N}_{n}\setminus {\mathcal I}$ or $e^{-\omega \cdot}f\in l^{1}({\mathbb Z} : [D(A_{i})])$ for all $i\in {\mathbb N}_{n}\setminus {\mathcal I}$ and $(S(v))_{v\in {\mathbb N}_{0}}\subseteq L(X)$ is a discrete $(k,C,B,(A_{i})_{1\leq i\leq n},(v_{i})_{1\leq i\leq n},{\mathbb N}_{n}\setminus {\mathcal I})$-existence family as well as $u=\Delta^{m_{n}}h$ for a certain sequence $h : {\mathbb Z} \rightarrow \bigcap_{i\in {\mathbb N}_{n}\setminus {\mathcal I}}D(A_{i}).$ Then we have:
\begin{align*}
B\bigl( \Delta^{m_{n}} h\bigr)(v)&=\sum_{i\in {\mathcal I}}A_{i}\Bigl(\Delta^{m_{n}-m_{i}}\Delta^{\alpha_{i}}_{W}h\Bigr)(v+v_{i})
\\&+\sum_{i\in{\mathbb N}_{n}\setminus {\mathcal I}}\Bigl(\Delta^{m_{n}-m_{i}}\Delta^{\alpha_{j}}_{W}A_{i}h\Bigr)(v+v_{i})+e^{-\omega v}(k \circ Cf)(v)+g_{\omega}(v),\ v\in {\mathbb Z}.
\end{align*}
\end{itemize}
\end{thm}

\begin{proof}
The proof simply follows by applying Theorem \ref{zajebq}, after observing that $(e^{-\omega v}S(v))_{v\in {\mathbb N}_{0}}$ is a discrete $(e^{-\omega \cdot}k(\cdot),C,B,(A_{i})_{1\leq i\leq n},(v_{i})_{1\leq i\leq n},{\mathcal I})$-existence family with the kernels $a_{i}(\cdot)$ replaced therein with the kernels $a_{i}(\cdot)e^{-\omega (\cdot-v_{i})}$ for $1\leq i\leq n.$
\end{proof}

As in many research articles published by now, this enables us to consider the almost periodic features of the function $e^{-\omega \cdot}u(\cdot),$ where $u(\cdot)$ solves \eqref{got}; cf. \cite{nova-selected} for more details on the subject.
The exponential boundedness of discrete $(k,B,C,(A_{i})_{1\leq i\leq n},(v_{i})_{1\leq i\leq n},{\mathcal I})$-existence family can be proved in many concrete situations; for example, if $v_{1}=...=v_{n}=0$, $k(\cdot)$ is exponentially bounded, the kernels $a_{i}(\cdot)$ are bounded for $1\leq i\leq n,$ $C={\rm I}$ and there exists a closed linear operator $A$ and the complex polynomials $P_{B}(\cdot),$ $ P_{i}(\cdot)$ ($1\leq i\leq n$) such that 
$B=P_{B}(A)$ and $A_{i}=P_{i}(A)$ for $1\leq i\leq n$.

Finally, we will present the following illustrative applications of Theorem \ref{zajebqe}:

\begin{example}\label{kucanie}
\begin{itemize}
\item[(i)] Suppose, for simplicity, that $\alpha>0,$ $m=\lceil \alpha \rceil$ and a closed linear operator $-A$ is a subgenerator of $(g_{\alpha},g_{\alpha})$-regularized $C$-resolvent family $(T(t))_{t\geq 0}$ such that there exist two real constants $M\geq 1$ and $c\in (0,1)$ such that
$\| T(t)\| \leq Me^{(1-c)t}, $ $t\geq 0;$ see \cite{FKP} for the notion and more details. Then it is not difficult to show that, for every $x\in X,$ the function $t\mapsto T(t)x,$ $t\geq 0$ is a mild solution of the abstract Cauchy problem (cf. \cite{kexue} for the case $1<\alpha \leq 2$; if $0<\alpha \leq 1,$ then we recover some known results about the problem \eqref{obor} with ${\mathcal A}=-A$):
$$
D_{t}^{\alpha}u(t)+Au(t)=0,\quad t>0.
$$
We define the sequence $(S(v))_{v\in {\mathbb N}_{0}}$ as in Example \ref{kucanje}. Then we have $\| S(v)\|\leq c^{-1}(c^{-1})^{v},$ $v\in {\mathbb N}_{0}$ so that $\sum_{v=0}^{+\infty}\| e^{-\omega v}S(v)\|<+\infty$ for $\omega>\ln (c^{-1}).$ Let us assume that the sequence $f: {\mathbb Z}\rightarrow X$ satisfies the assumptions (a) and (c) from the formulation of Theorem \ref{zajebqe}. If we define $u(\cdot)$ as in Example \ref{kucanje}, then the equation \eqref{je} holds with the sequence $g(\cdot)$ replaced with the sequence $g_{\omega}(\cdot),$ obtained by replacing the forcing term $f(\cdot)$ in \eqref{getri} by $e^{-\omega \cdot}f(\cdot).$ Arguing in the same way as in Example \ref{kucanje}, we get that
$$
\Bigl[\Delta^{\alpha}_{W}u\Bigr](v)+Au(v+m)=-g_{\omega}(v),\quad v\in {\mathbb Z}.
$$
\item[(ii)] In this part, we will transfer the conclusions established in (i) for the equations with the Caputo fractional derivatives. Suppose that $\alpha>0,$ $m=\lceil \alpha \rceil$ and a closed linear operator $-A$ is a subgenerator of a global $(g_{\alpha},C)$-regularized resolvent family $(T(t))_{t\geq 0}$ such that there exist two real constants $M\geq 1$ and $c\in (0,1)$ with
$\| T(t)\|  \leq Me^{(1-c)t}, $ $t\geq 0;$ see \cite{bajlekova}, \cite{knjigaho} and \cite{FKP} for the notion and the corresponding examples. Then we know that, for every $x\in X,$ the function $t\mapsto T(t)x,$ $t\geq 0$ is a mild solution of the abstract Cauchy problem
$$
\mathbf D_t^{\alpha}u(t)+Au(t)=0,\quad t>0; \ u(0)=Cx,\ u^{(j)}(0)=0,\ 1\leq j\leq m-1.
$$
Using the formula established in \cite{multi-term}:
\begin{align*}
\int^{+\infty}_{0}& e^{-t}\frac{t^{v+m_{n}}}{(v+m_{n})!}{\mathbf D}_{t}^{\alpha}u(t)\, dt=\bigl( \Delta^{\alpha}y \bigr)(v)
\\&+\frac{(-1)^{v+m+1}}{(v+m)!}\sum_{k=0}^{m-1}u^{(k)}(0)(\alpha-1-k)\cdot ...\cdot (\alpha-k-v-m),\ v\in {\mathbb N}_{0},
\end{align*}
we can similarly prove that
$$
0=\bigl( \Delta^{\alpha}S\bigr)(v)x+AS(v+m)x+k(v)Cx,\quad v\in {\mathbb N}_{0},\ x\in X,
$$
where 
$$
k(v)=\frac{(-1)^{v+m+1}}{(v+m)!}(\alpha-1)\cdot ...\cdot (\alpha-v-m),\quad v\in {\mathbb N}_{0}.
$$
Define the sequence $(S(v))_{v\in {\mathbb N}_{0}}$ as in (i). Then $(S(v))_{v\in {\mathbb N}_{0}}$ is a discrete
$(k,0,C,({\rm I},A),(0,m),\emptyset)$-existence family, with 
$$
a_{2}(v)=\sum_{j=0}^{m}(-1)^{m-j}\binom{m}{j}k^{m-\alpha}\bigl(v+j-m\bigr)\ \ \mbox{ and } \ \  a_{1}(v)=k^{0}(v),\quad v\in  {\mathbb N}_{0}.
$$
Arguing as in Example \ref{kucanje}, with the same notion of the term $g_{\omega}(\cdot),$ we obtain that
$$
\Bigl[\Delta^{\alpha}_{W}u\Bigr](v)+Au(v+m)=-g_{\omega}(v)-e^{\omega v}(k\circ Cf)(v),\quad v\in {\mathbb Z},
$$
provided that the sequence $f: {\mathbb Z}\rightarrow X$ satisfies the assumptions (a) and (c) from the formulation of Theorem \ref{zajebqe}.
\item[(iii)] Suppose, finally, that $0\leq \alpha_{1}<\alpha_{2}<....<\alpha_{n},$ $k\in {\mathbb N}_{m_{n}-1}^{0}$ and the resolvent operator family $(T(t))_{t\geq 0}$ satisfies that there exist two real constants $M\geq 1$ and $c\in (0,1)$ with
$\| T(t)\|  \leq Me^{(1-c)t}, $ $t\geq 0$ as well as that, for every $x\in X,$ the function $t\mapsto T(t)x,$ $t\geq 0$ is a mild solution of the abstract Cauchy problem
$$
A_{n}\mathbf D_t^{\alpha_{n}}u(t)+...+A_{1}\mathbf D_t^{\alpha_{1}}u(t)=0,\ t>0; \ u^{(j)}(0)=\delta_{jk}Cx,\ 0\leq j\leq m_{n}-1;
$$
see \cite{knjigaho} and \cite{FKP} for the notion and the corresponding examples. Set $A_{k}:=\{j\in {\mathbb N}_{n-1}^{0} : m_{j}-1\geq k\}$ and
\begin{align*}
k(v):=\sum_{j\in A_{k}}\frac{(-1)^{v+m_{j}+1}}{(v+m_{j})!}(\alpha-1-k)\cdot ....\cdot (\alpha-1-k-m_{j}).
\end{align*}
We define the sequence $(S(v))_{v\in {\mathbb N}_{0}}$ as in Example \ref{kucanje}. Then we have
\begin{align*}
A_{n}\bigl( \Delta^{\alpha_{n}}S(\cdot)x \bigr)(v)&+\sum_{s=1}^{n-1}A_{s}\bigl( \Delta^{\alpha_{s}}S(\cdot)x \bigr)(v+m_{n}-m_{s})
=-k(v)Cx,\ v\in {\mathbb N}_{0},\ x\in X,
\end{align*}
so that $(S(v))_{v\geq 0}$ is a discrete\\ $(k,0,C,(A_{i})_{1\leq i\leq n},(m_{n}-m_{1},...,m_{n}-m_{n-1},0),\emptyset)$-existence family. If the sequence $f: {\mathbb Z}\rightarrow X$ satisfies the assumptions (a) and (c) from the formulation of Theorem \ref{zajebqe}, then the foregoing argumentation shows that
\begin{align*}
\Bigl[\Delta^{\alpha_{n}}_{W}A_{n}u\Bigr](v)&+\Bigl[\Delta^{\alpha_{n-1}}_{W}A_{n-1}u\Bigr]\bigl(v+m_{n}-m_{n-1}\bigr)
\\&+...+\Bigl[\Delta^{\alpha_{1}}_{W}A_{1}u\Bigr]\bigl(v+m_{n}-m_{1}\bigr)=-g_{\omega}(v)-e^{\omega v}(k\circ Cf)(v),\ v\in {\mathbb Z},
\end{align*}
where we define $g_{\omega}(\cdot)$ as in (ii).
\end{itemize}
\end{example}

\section{Conclusions and final remarks}\label{rema}

In this paper, we have examined
the abstract non-scalar Volterra difference equations. We have employed the Poisson like transforms and, as a special case of our analysis, we have investigated the existence, uniqueness and almost periodicity to the abstract multi-term fractional difference equations with Weyl fractional derivatives.

\end{document}